\documentclass[aap,preprint]{imsart}

\usepackage{amsthm,amsmath}
\usepackage[round]{natbib}
\usepackage{hyperref}
\usepackage{verbatim}
\usepackage{amssymb}
\usepackage{mathrsfs} 
\usepackage{bbm} 

\startlocaldefs

\newcounter{subsection1}[section]
\newtheorem{defn}[subsection1]{Definition}
\newtheorem{lemma}[subsection1]{Lemma}
\newtheorem{prop}[subsection1]{Proposition}
\newtheorem{theorem}[subsection1]{Theorem}
\newtheorem{remark}[subsection1]{Remark}

\numberwithin{equation}{section} 
\numberwithin{subsection1}{section}

\def\to{\rightarrow} 

\def\mc{\mathcal} 
\def\mb{\mathbb} 
\def\v{\mathbf} 
\def\E{\mb{E}} 
\def\P{\mb{P}}
\def\R{\mb{R}} 
\def\N{\mb{N}}

\def\Z{\mb{Z}}

\def\~{\sim}
\def\-{\,:\,} 

\def\|{\,|\,} 

\def\1{\mathbbm{1}}

\def\slfv{S$\Lambda$FV}
\def\slfvs{S$\Lambda$FVS}
\def\l{\left}
\def\r{\right}

\def\IH1{\textit{($IH$)$_1$}}

\newcommand \ind {\mathbf{1}}

\endlocaldefs

\begin{document}

\allowdisplaybreaks

\begin{frontmatter}

\title{Branching Brownian motion and Selection in the Spatial $\Lambda$-Fleming-Viot Process}
\runtitle{BBM and Selection in the {\slfv} Process}


\begin{aug}

\author{\fnms{Alison} \snm{Etheridge}\ead[label=e1]{etheridg@stats.ox.ac.uk}\thanksref{t1},}
\author{\fnms{Nic} \snm{Freeman}\corref{}\ead[label=e2]{n.p.freeman@sheffield.ac.uk},}
\author{\fnms{Sarah} \snm{Penington}\ead[label=e3]{sarah.penington@sjc.ox.ac.uk}\thanksref{t2}}
\and
\author{\fnms{Daniel} \snm{Straulino}\ead[label=e4]{danielstraulino@gmail.com}\thanksref{t3}}

\affiliation{University of Oxford, University of Sheffield and Heilbronn Institute for Mathematical Research}

\thankstext{t1}{supported in part by EPSRC Grant EP/I01361X/1}
\thankstext{t2}{supported by EPSRC DTG EP/K503113/1}
\thankstext{t3}{supported by CONACYT}

\address{Alison Etheridge \\ Department of Statistics \\ University of Oxford \\  24-29 St Giles \\ Oxford \\ England \\ \printead{e1}}
\address{Nic Freeman \\ School of Mathematics and Statistics \\ University of Sheffield \\ Hounsfield Road \\ Sheffield \\ England \\\printead{e2}}
\address{Sarah Penington \\ Department of Statistics \\ University of Oxford \\ 24-29 St Giles \\ Oxford \\ England \\ \printead{e3}}
\address{Daniel Straulino \\ Department of Statistics \\ University of Oxford \\ 24-29 St Giles \\ Oxford \\ England \\ \printead{e4}}

\end{aug}

\runauthor{A.~Etheridge, N.~Freeman, S.~Penington, D.~Straulino}

\begin{abstract}

We ask the question ``when will natural selection on a gene in a spatially structured population
cause a detectable trace in the patterns of genetic variation observed in the contemporary 
population?''. We focus on the situation in which `neighbourhood size', that is the effective 
local population density, is small.
The genealogy relating individuals in a sample from the population is embedded in a spatial
version of the ancestral selection graph and through applying a diffusive scaling to this object we
show that whereas in dimensions at least three, selection is barely impeded by the spatial
structure, in the most relevant dimension, $d=2$, selection must be stronger (by a factor of
$\log(1/\mu)$ where $\mu$ is the neutral mutation rate) if we are to have a chance of detecting it.
The case $d=1$ was handled in \cite{EFS2015}.

The mathematical interest is that although the system of branching and coalescing lineages that
forms the ancestral selection graph converges to a branching Brownian motion, this reflects a 
delicate balance of a branching rate that grows to infinity and the instant annullation of almost
all branches through coalescence caused by the strong local competition in the population.

\end{abstract}



\end{frontmatter}

\section{Introduction}\label{intro}

Our aims in this work are two-fold. 
On the one hand, we address a question of interest in population genetics: 
when will the action of natural selection on a gene in a spatially structured population 
cause a detectable trace in the patterns of genetic variation observed in the contemporary population?
On the other hand, we investigate some of the rich structure underlying mathematical models 
for spatially evolving populations and, in particular, the systems of interacting random walks that, 
as dual processes (corresponding to ancestral lineages of the model), describe the genetic relationships
between individuals sampled from those populations. 

Since the seminal work of \cite{fisher:1937}, a large literature has developed that investigates
the interaction of natural selection with the spatial structure of a population. 
Traditionally, the deterministic action of migration and selection is approximated
by what we now call the Fisher-KPP equation and predictions from that equation are compared to data.
However, many important questions depend on how selection and migration interact with a
third force, the stochastic fluctuations known as random genetic drift, and this poses significant 
new mathematical challenges. 

For the most part, random drift is modelled through Wright-Fisher noise 
resulting in a stochastic PDE as a model for the evolution of gene frequencies $w$:
$$\frac{\partial w}{\partial t}
=m\Delta w- s w (1-w)+\sqrt{\gamma w(1-w)}\dot{\mathcal W}$$
(for suitable constants $m$, $s$ and $\gamma$),
where $\mathcal W$ is space-time white noise.
This stochastic Fisher-KPP equation has been extensively studied, 
see, for example, \cite{mueller/mytnik/quastel:2008} and references therein.  However, from a modelling perspective it has  
two immediate shortcomings. First, it only makes sense in one spatial dimension.  This is generally
overcome by artificially subdividing the population, and thus replacing the stochastic PDE by a system
of stochastic ordinary differential equations, coupled through migration. The second problem is that, 
in deriving the equation, one allows the `neighbourhood size' to tend to infinity.
We shall give a precise definition of neighbourhood size in Section~\ref{model}. Loosely, it is
inversely proportional to the probability that two individuals sampled from sufficiently close
to one another had a common parent in the previous generation and small neighbourhood size
corresponds to strong genetic drift.  
It is understanding
the implications of dropping this (usually implicit) assumption of unbounded neighbourhood size that motivated the work presented here. 

Our starting point will be the Spatial $\Lambda$-Fleming-Viot process with selection ({\slfvs}), which 
(along with its dual) was introduced and constructed in \cite{EVY2014}. The dynamics of both the {\slfvs} 
and its dual are driven by a Poisson Point Process of `events' (which model reproduction or
extinction and recolonisation in the population)
and will be described in detail in Section~\ref{model}. The advantage of this model is that it
circumvents the need to subdivide the population in higher dimensions. However, since our proof is 
based on an analysis of the branching and coalescing system of random walkers that describes the 
ancestry of a sample from the population, it would be straightforward to modify it to apply to, for
example, an individual based model in which a fixed number of individuals reside at each point of a 
$d$-dimensional lattice.

In classical models of population genetics, in which there is no spatial structure, we generally think of
population size as setting the timescale of evolution of frequencies of different genetic types. 
Evidently that makes no sense in our setting. However (even in the classical setting), as we explain in
more detail in Section~\ref{biology},
if natural selection is to leave a distinguishable trace in contemporary patterns of genetic variation, then
a sufficiency of neutral mutations must fall on the genealogical trees relating 
individuals in a sample. Thus, in fact, it is the neutral mutation rate which sets the timescale 
and, since mutation rates are very low, this leads us to consider
scaling limits. 

In \cite{EVY2014}, scaling limits of the (forwards in time) {\slfvs} were considered in which the 
neighbourhood size tends to infinity. In that case, the classical Fisher-KPP equation and, in one 
spatial dimension, its stochastic analogue are recovered. The dual process of branching and coalescing 
lineages converges to branching Brownian motion, with coalescence of lineages (in one dimension) at a rate determined by 
the local time that they spend together. 
In this article we consider scaling limits in the (very different) regime in which neighbourhood 
size remains finite. In this context the interaction between genetic drift and spatial structure becomes much more important and, 
in contrast to \cite{EVY2014}, it is the dual process which proves to be the 
more analytically tractable object.

We shall focus on the most biologically relevant case of two spatial dimensions. The case of
one dimension was discussed in \cite{EFS2015}. The main interest there is mathematical: the dual
process of branching and coalescing ancestral lineages, suitably scaled, converges to the Brownian net.
However, the scaling required to obtain a non-trivial limit reveals a strong effect of the spatial structure. 
Here we shall identify the corresponding scalings in dimensions $d\geq 2$.
Whereas in \cite{EVY2014}, the scaling of the selection coefficient
is independent of spatial dimension and, indeed, mirrors that for unstructured populations, for bounded 
neighbourhood size this is no longer the case. In $d=1$ and $d=2$ the scaling of the selection 
coefficient required to obtain a non-trivial limit reflects strong local competition.

Our main result, Theorem~\ref{result d>1}, is that under these (dimension-dependent) scalings,
the scaled dual process converges to a branching Brownian motion. For $d\geq 3$ this is 
rather straightforward, but in two dimensions things are much more delicate.
The mathematical interest of our result is that in $d=2$,
under our scaling, the rate of branching of ancestral lineages 
explodes to infinity but, crucially, all except finitely many branches are instantaneously annulled through coalescence.
That this finely balanced 
picture produces a non-degenerate limit results from a combination of the  
failure of two dimensional Brownian motion to hit points and the strong (local) interactions of the approximating random walks,
which cause coalescence. 

From a biological perspective, the main interest is that, in contrast to
the infinite neighbourhood size limit, here we see a strong effect of spatial dimension in our results. 
When neighbourhood size is very big, the probability of fixation for an advantageous genetic type,
i.e. the probability that the genetic type establishes and sweeps 
through the entire population, is not affected by spatial structure. When neighbourhood size is small, in (one and)
two spatial dimensions, selection has to be much stronger to leave a detectable trace  
than in a population with no spatial structure. 
Indeed, local establishment is no longer a guarantee of eventual fixation. 

The rest of the paper is laid out as follows. In Section~\ref{model} we describe the {\slfvs} and the dual
process of branching and coalescing random walks, state our
main result and provide a heuristic argument that explains our choice of scalings.
In Section~\ref{biology} we place our findings in the context of previous work on selective sweeps in
spatially structured populations and  
in Section~\ref{proof} we prove our result.

~

\noindent
{\bf Acknowledgements}

Our results (with different proofs) form part of the DPhil thesis of the last author. We would like to thank
the examiners, Christina Goldschmidt and Anton Wakolbinger, for their careful reading of the thesis and 
detailed feedback.
We would also like to thank the two anonymous referees for their careful reading of the paper and valuable comments.

\section{The model and main result}
\label{model}

\subsection{The model}

To motivate the definition of the {\slfvs}, it is convenient to recall (a very special case of)
the model without selection, 
introduced in \cite{E2008, BEV2010}.  We shall call it the {\slfv} to emphasize that 
selection is not acting. We proceed informally, only carefully specifying the state space and conditions
that are sufficient to guarantee existence of the process when we define the {\slfvs} itself in Definition~\ref{slfvdefn}.
The interested reader can find much more general conditions under which the {\slfv} exists in
\cite{EK2014}. 

We restrict ourselves to the case of just two genetic types,
which we denote $a$ and $A$, and we suppose that the population is evolving in $\R^d$. 
It is convenient to index time by the whole real line.
At each time $t$, the 
random function $\{w_t(x),\, x\in \R^d\}$ 
is defined, up to a Lebesgue null set of $\R^d$, by
\begin{equation}
\label{defn of w}
w_t(x):= \hbox{ proportion of type }a\hbox{ at spatial position }x\hbox{ at time }t.
\end{equation}
The dynamics are driven by a
Poisson point process $\Pi$ on $\R\times \R^d\times \R_+\times (0,1]$. Each point 
$(t,x,r,u)\in\Pi$ specifies a reproduction event which will affect that part of the population 
at time $t$ which lies within
the closed ball $\mc{B}_r(x)$ of radius $r$ centred on the point $x$. 
First the location $z$ of the parent of the event is chosen uniformly at random from $\mc{B}_r(x)$.
All offspring inherit the type $\alpha$ of the parent which is determined by 
$w_{t-}(z)$; that is,
with probability $w_{t-}(z)$ all offspring will be type $a$, otherwise they will be $A$. 
A portion $u$ of the population within the ball is then replaced by offspring so that
$$w_t(y)=(1-u)w_{t-}(y)+u \1_{\{\alpha=a\}},\qquad\forall y\in \mc{B}_r(x).$$
The population outside the ball is unaffected by the event.
We sometimes call $u$ the impact of the event. 

Under this model, the time reversal of the same Poisson Point 
Process of events governs the ancestry of a sample
from the population.  Each ancestral lineage that lies in the region affected by an event has a probability
$u$ of being among the offspring of the event, in which case, as we trace backwards in 
time, it jumps to the location of 
the parent, which is sampled uniformly from the region. In this way, ancestral lineages
evolve according to (dependent) compound Poisson processes and lineages can coalesce when affected by
the same event. All lineages affected by an event inherit the type of the parent of that event.

\begin{remark}
In \cite{EK2014}, the {\slfv} and its dual are constructed simultaneously on 
the same probability space, through a lookdown construction, as the limit of an 
individual based model, and so the dual process just described really can be interpreted as 
tracing the ancestry of individuals in a sample from the population. 
\end{remark} 

We are now in a position to define the neighbourhood size.
\begin{defn}
Write $\sigma^2$ for the variance of the first coordinate of the location of a single ancestral lineage after 
one unit of time and $\eta(x)$ for the instantaneous rate of coalescence of two lineages 
that are currently at a separation $x\in\R^d$. Then the {\em neighbourhood size}, ${\cal N}$ is
given by 
$${\cal N}=\frac{2dC_d\sigma^2}{\int_{\R^d}\eta(x)dx},$$
where $C_d$ is the volume of the unit ball in $\R^d$.
\end{defn}
Neighbourhood size is used in biology to quantify the local number of breeding individuals in a continuous population; see \cite{barton/etheridge/kelleher/veber:2013a} for a derivation of this formula.
If we assume that the impact is the same for all events, then 
the impact is inversely proportional to the 
neighbourhood size, see \cite{barton/etheridge/kelleher/veber:2013a}.

There are very many different ways in which to introduce selection into the {\slfv}. 
Our approach here is a simple adaptation of that
adopted in classical models of population genetics.
The parental type in the {\slfv} is a uniform pick 
from the types in the region affected by the event. We can introduce
a small advantage to individuals of type $A$ by choosing the parent in a weighted way. Thus if, 
immediately before reproduction, the
proportion of type $a$ individuals in the region affected by the event is $\overline{w}$, then 
the offspring will be type $a$ with probability $\overline{w}/(1+\v{s}(1-\overline{w}))$. 
We say that the relative
fitnesses of types $a$ and $A$ are $1$ and $1+\v{s}$ respectively and refer to $\v{s}$ as the selection
coefficient. We are interested only in small values of $\v{s}$ and so we expand
$$\frac{\overline{w}}{1+\v{s}(1-\overline{w})}=\overline{w} \{1-\v{s}(1-\overline{w})\}
+{\mathcal O}(\v{s}^2)
=(1-\v{s})\overline{w}+\v{s}\overline{w}^2 +{\mathcal O}(\v{s}^2).$$
We shall regard $\v{s}^2$ as being negligible. 
We can then think of each event, independently, as being a `neutral' event with
probability $(1-\v{s})$ and a `selective' event with probability $\v{s}$.
Reproduction during neutral events is exactly as before, but during selective events, we sample two
{\em potential} parents; only if both are type $a$ will the offspring be of type $a$. 

Let us now give a more precise definition of the {\slfvs}.  We retain the notation of~(\ref{defn of w}).
A construction of an appropriate state space for $x\mapsto w_t(x)$ can be found in \cite{veber/wakolbinger:2013}. 
Using the identification
$$
\int_{\R^d\times \{a,A\}} f(x,\kappa) M(dx,d\kappa) = \int_{\R^d} \big\{w(x)f(x,a)+ (1-w(x))f(x,A)\big\}\, dx,
$$
this state space is in one-to-one correspondence with the space
${\cal M}_\lambda$ of measures on $\R^d\times\{a,A\}$ with `spatial marginal' Lebesgue measure,
which we endow with the topology of vague convergence. By a slight abuse of notation, we also denote the
state space of the process $(w_t)_{t\in\R}$ by 
${\cal M}_\lambda$.

\begin{defn}[{\slfv} with selection ({\slfvs})]
\label{slfvdefn}
Fix $\mc{R}\in(0,\infty)$.  Let $\mu$ be a finite measure on $(0,\mc{R}]$ and, for each $r\in (0,\mc{R}]$,
let $\nu_r$ be a probability measure on $(0,1]$.
Further, let $\Pi$ be a Poisson point process on 
$\R\times \R^d\times (0,\mc{R}]\times (0,1]$ with intensity measure 
\begin{equation}\label{slfvdrive}
dt\otimes dx\otimes \mu(dr)\nu_r(du).
\end{equation}
The {\em spatial $\Lambda$-Fleming-Viot process with selection} ({\slfvs})
driven by \eqref{slfvdrive} is the ${\cal M}_\lambda$-valued process $(w_t)_{t\in\R}$ with dynamics given as follows.

If $(t,x,r,u)\in \Pi$, a reproduction event occurs at time $t$ within the closed ball $\mc{B}_r(x)$ of radius $r$ centred on $x$. With probability $1-\v{s}$ the event is {\em neutral}, in which case:
\begin{enumerate}
\item Choose a parental location $z$ uniformly at random within 
$\mc{B}_r(x)$, 
and a parental type, $\alpha$, according to $w_{t-}(z)$, that is
$\alpha=a$ with probability $w_{t-}(z)$ 
and $\alpha=A$ with probability $1-w_{t-}(z)$.
\item For every $y\in \mc{B}_r(x)$, set 
$w_t(y) = (1-u)w_{t-}(y) + u\ind_{\{\alpha=a\}}$.
\end{enumerate}
With the complementary probability $\v{s}$ the event is {\em selective},
in which case:
\begin{enumerate}
\item Choose two `potential' parental locations $z,z'$ independently and uniformly at random within $\mc{B}_r(x)$, and at each of these sites `potential' parental types 
$\alpha$, $\alpha'$, according to $w_{t-}(z), w_{t-}(z')$ respectively.
\item For every $y\in \mc{B}_r(x)$ set 
$w_t(y) = (1-u)w_{t-}(y) + u\ind_{\{\alpha =\alpha'=a\}}$.
Declare the parental location to be $z$ if $\alpha=\alpha'=a$ or
$\alpha=\alpha'=A$ and to be $z$ (resp. $z'$) if $\alpha=A,\alpha'=a$
(resp. $\alpha=a, \alpha'=A$).
\end{enumerate}
\end{defn}
This is a very special case of the {\slfvs} introduced in 
\cite{EVY2014}. 

We are especially concerned with the dual process of the {\slfvs}.
Whereas in the neutral case we can always identify the distribution of the 
location of the parent of each event, without any additional information on the distribution of types
in the region, 
now, at a selective event, we are unable to identify which of the `potential parents' is the
true parent of the event without knowing their types. These can only be established by
tracing further into the past.
The resolution is to follow all
{\em potential} ancestral lineages backwards in time. This results in a system of 
branching and coalescing walks.

As in the neutral case, the
dynamics of the dual are driven by the same Poisson point process of events, $\Pi$, 
that drove the forwards in 
time process. The distribution of this Poisson point process is invariant under time reversal and so we shall 
abuse notation by reversing the direction of time when discussing the dual.

We suppose that at time $0$ (which we think of as `the present'), we sample $k$ individuals from locations 
$x_1,\ldots ,x_k$ and we write $\xi_s^1,\ldots ,\xi_s^{N_s}$ for the locations of the $N_s$ 
{\em potential ancestors} that make up our  dual at time $s$ before the present. 

\begin{defn}[Branching and coalescing dual]
\label{dualprocessdefn}
The branching and coalescing dual process $(\Xi_t)_{t\geq 0}$ driven by $\Pi$ is the
$\bigcup_{m\geq 1}(\R^d)^m$-valued Markov process with dynamics defined as follows:  
at each event $(t,x,r,u)\in \Pi$, with probability $1-\v{s}$, the event is neutral:
\begin{enumerate}
\item For each $\xi_{t-}^i\in \mc{B}_r(x)$, independently mark 
the corresponding 
lineage with probability $u$;
\item if at least one lineage is marked, 
all marked lineages disappear and are replaced by a single
lineage, whose location at time $t$ is drawn uniformly at random 
from within $\mc{B}_r(x)$.
\end{enumerate}
With the complementary probability $\v{s}$, the event is selective:
\begin{enumerate}
\item For each $\xi_{t-}^i\in \mc{B}_r(x)$, independently mark the corresponding
lineage with probability $u$;
\item if at least one lineage is marked,
all marked lineages disappear and are replaced by {\em two} lineages, whose locations at time $t$ are drawn independently and uniformly from
within $\mc{B}_r(x)$.
\end{enumerate}
In both cases, if no lineages are marked, then nothing happens.
\end{defn}

Since we only consider finitely many initial individuals in the sample, 
and the jump rate of the dual is bounded by a linear function of the number of potential ancestors,
this description gives rise to a well-defined process. 

This dual process is the analogue for the {\slfvs} of the Ancestral Selection Graph (ASG), 
introduced in the companion papers \cite{krone/neuhauser:1997, neuhauser/krone:1997}, 
which describes all the potential ancestors 
of a sample from a population evolving according to the Wright-Fisher diffusion with selection. 
Perhaps the simplest way of expressing the duality between the {\slfvs} and the branching and
coalescing dual process is to observe that all the individuals in our sample are of type $a$
if and only if all potential ancestral lineages are of type $a$ at any time $t$ in the past. 
This is analogous to the {\em moment duality} between the ASG and the Wright-Fisher 
diffusion with selection.
However, to state this formally for the {\slfvs}, we would need to be able to identify 
$\E[\prod_{i=1}^n w_t(x_i)]$ for any choice of points $x_1,\ldots ,x_n\in\R^d$. The difficulty 
is that, just as in the neutral case, the {\slfvs} $w_t(x)$ is only defined at Lebesgue almost every point 
$x$ and so we have to be satisfied with a `weak' moment duality.
\begin{prop}\label{prop: dual}[\cite{EVY2014}]
The spatial $\Lambda$-Fleming-Viot process with selection is 
dual to the process $(\Xi_t)_{t\geq 0}$ in the sense that 
for every $k\in \N$ and $\psi\in C_c((\R^d)^k)$, we have
\begin{align}
\E_{w_0}\bigg[\int_{(\R^d)^k} & \psi(x_1,\ldots,x_k)\bigg\{\prod_{j=1}^k w_t(x_j)\bigg\}\, dx_1\ldots dx_k\bigg] \nonumber\\
& = \int_{(\R^d)^k} \psi(x_1,\ldots,x_k)\E_{\{x_1,\ldots,x_k\}}\bigg[\prod_{j=1}^{N_t} w_0\big( \xi_t^j\big)\bigg]\, dx_1 \ldots dx_k. \label{dual formula}
\end{align}
\end{prop}

\subsection{The main result}
\label{results}

Our main result concerns a diffusive rescaling of the dual process of Definition~\ref{dualprocessdefn}
and so from now on it will be convenient if

\centerline{\em forwards in time refers to forwards for the dual process.}

We shall take the impact parameter, $u$, to be a fixed number in $(0,1]$ (i.e.~$\nu_r=\delta_u$ for all $r$). 
In fact, the same arguments work when $u$ is allowed to be random, as long as $\int_{\mc R'}^{\mc R}\int_0^1 u \nu_r (du)\mu(dr)>0$ for some $0<\mc R'< \mc R$, but this would make our proofs notationally cumbersome.

Let us describe the scaling more precisely. Suppose that $\mu$ is a finite measure on $(0,\mc{R}]$.
We shall assume for convenience that $\mc R$ is defined in such a way that for any $\delta>0$, $\mu((\mc R -\delta , \mc R])>0$.
For each $n\in\N$, define the measure $\mu^n$ by
$\mu^n(B)=\mu(n^{1/2}B),$
for all Borel subsets $B$ of $\R_{+}$. It will be convenient to write $\mc{R}_n=\mc{R}/\sqrt{n}$.
At the $n$th stage of the rescaling, our rescaled dual is driven by
the Poisson point process $\Pi^n$ on $\R\times \R^d \times (0,\mc{R}_n]$ 
with intensity
\begin{equation}\label{rescalingeq}
n\,dt \otimes n^{d/2}\,dx \otimes \mu^n(dr).
\end{equation}
This corresponds to rescaling space and time from $(t,x)$ to $(n^{-1}t,n^{-1/2}x)$.
Importantly, we do not scale the impact $u$.
Each event of $\Pi^n$, independently, is neutral with probability $1-\v{s}_n$
and selective with probability $\v{s}_n$, where 
\begin{equation}\label{sdef}
\v{s}_n=
\begin{cases}
\frac{\log n}{n} & d=2, \\
\frac{1}{n} & d\geq 3.
\end{cases}
\end{equation}
In \cite{EFS2015} it was shown that in $d=1$, one should take $\v{s}_n=1/\sqrt{n}$.

Although not obvious for the {\slfvs} itself, when considering the dual process it is not hard to understand why 
the scalings \eqref{rescalingeq} and \eqref{sdef} should lead to a non-trivial limit. 

If we ignore the selective events, 
then a single ancestral lineage 
evolves as a pure jump process which is
homogeneous in both space and time. Write $V_r$ for the volume
of $\mc{B}_r(0)$.
The rate at which the lineage jumps from
$y$ to $y+z$ can be written
\begin{equation}
\label{jump of size z}
m_n(dz)=nu\int_0^{\mc{R}_n}n^{d/2}\frac{V_r(0,z)}{V_r}\mu^n(dr)\,dz,
\end{equation}
where $V_r(0,z)$ is the volume of ${\mc B}_r(0)\cap {\mc B}_r(z)$.
To see this, by spatial homogeneity, we may take the lineage to be at the 
origin in $\R^d$ before the jump, and then, in order for it to jump to
$z$, it must be affected by an event that covers both $0$ and $z$. If the
event has radius $r$, then the 
volume of possible centres, $x$, 
of such events is $V_r(0,z)$ and so the intensity with which such a centre
is selected is
$n\,n^{d/2}V_r(0,z)\mu^n(dr)$. 
The parental location is
chosen uniformly from the ball $\mc{B}_r(x)$, so the probability that 
$z$ is chosen as the parental location is $dz/V_r$ and the probability
that our lineage is actually affected by the event is $u$.
Combining these yields~\eqref{jump of size z}.

The total rate of jumps is
\begin{eqnarray}
\int_{\R^d}m_n(dz)&=&\int_0^{\mc{R}_n}nu\,n^{d/2}\frac{1}{V_r}
\int_{\R^d}\int_{\R^d}\1_{|x|<r}\1_{|x-z|<r}dx\,dz\,\mu^n(dr)
\nonumber\\
&=&\int_0^{\mc{R}_n}nu\,n^{d/2}V_r\mu^n(dr)\nonumber \\
&=&n u V_1\int_0^{\mc{R}}r^d\mu(dr)=\Theta(n),\label{jump rate}
\end{eqnarray}
and the size of each jump is $\Theta(n^{-1/2})$ and so in the limit a single lineage will evolve according to
a (time-changed) Brownian motion. 

Now, consider what happens at a selective event. The two new lineages
are created at a separation of order $1/\sqrt{n}$. If we are to see both lineages in the limit then they must move apart 
to a separation of order $1$ (before, possibly, coalescing back together). Ignoring possible interactions with other lineages, the probability that a pair of lineages 
makes such an excursion is of order $1$ in $d\geq 3$, order $1/\log n$ in $d=2$ and order $1/\sqrt{n}$ in $d=1$. 
Therefore, in order to have a positive probability of seeing branching in the scaling limit, in $d\geq 3$ we only 
need that there are a positive number of selective events in unit (rescaled) time, and, for this, it is enough that $\v{s}_n$ 
is order $1/n$. However, for $d=2$, we need order $\log n$ branches before we expect to find one that is visible to us, 
hence the choice $\v{s}_n=\log n/n$. 



\begin{remark}
Our scaling mirrors that described in \cite{durrett/zahle:2007} for a model of a {\em hybrid zone}
(by which we mean a region in which we see both genetic types) which
develops around a boundary between two regions, in one of which type $a$ individuals are selectively favoured and in the 
other of which type $A$ individuals are selectively favoured.  In contrast to our continuum setting, 
their model is a spin system in which exactly one individual lives at each point of $\Z^d$.
\end{remark}

Before formally stating our main result, we need some notation. We shall denote by $\text{BBM}(p, V)$ binary branching 
Brownian motion started from the point $p\in\R^d$, with branching rate $V$ and diffusion constant given by
\begin{equation}
\label{first sigma squared}
\sigma ^2 =
\tfrac{1}{d}\int_{\R^d}|z|^2 m^n(dz)
= \tfrac{1}{d}\int_{\R^d}\int_0^\infty |z|^2 u \frac{V_r (0,z)}{V_r}\mu(dr)\,dz
\end{equation}
where $m^n(dz)$ is defined in \eqref{jump of size z}.
In other words, during their lifetime, which is exponentially distributed with parameter $V$, individuals
follow $d$-dimensional Brownian motion with diffusion constant $\sigma^2$, at the end of which they die, leaving behind
at the location where they died exactly two offspring.
We view $\text{BBM}(p,V)$ as a set of (continuous) paths, each starting at $p$, with precisely one path following each possible distinct sequence of branches. 

Similarly, we write $\mc{P}^{(n)}(p)$ for the dual process of Definition~\ref{dualprocessdefn},
rescaled as in \eqref{rescalingeq} and \eqref{sdef},
started from a single individual at the point $p\in\R^d$ and
viewed as a collection of paths. Each path traces out a `potential ancestral lineage', defined exactly as
the ancestral lineages in the neutral case except that at each selective event, if a lineage is affected then 
it jumps to the location of (either) one of the `potential parents'. Precisely one potential ancestral lineage follows
each possible route through the branching and coalescing dual process.

We define the events
\begin{align}
\mc{D}_n(\epsilon, T)=&\l\{\forall l\in \mc{P}^{(n)}(p),\;\exists l'\in \text{BBM}(p,V):\;\sup\limits_{t\in[0,T]}|l(t)-l'(t)|\leq\epsilon\r\},\notag\\
\mc{D}'_n(\epsilon, T)=&\l\{\forall l\in \text{BBM}(p,V),\;\exists l'\in \mc{P}^{(n)}(p):\;\sup\limits_{t\in[0,T]}|l(t)-l'(t)|\leq\epsilon\r\}.\label{Devents}
\end{align}

\begin{theorem}\label{result d>1}
Let $d\geq 2$. There exists $V\in (0,\infty)$ such that the following holds.
Let $T<\infty$, $p\in \R^2$; then given $\epsilon >0$, there exists
$N\in\N$ such that, for all $n\geq N$ there is a coupling between $\text{BBM}(p,V)$ and 
$\mc{P}^{(n)}(p)$ with
$\P\l[\mc{D}_n(\epsilon,T)\cap\mc{D}'_n(\epsilon,T)\r]\geq 1-\epsilon.$
\end{theorem}

We will give a proof of Theorem \ref{result d>1} only for $d=2$.
The case $d\geq 3$ follows from a simplified version of the $2$-dimensional proof presented here.

\subsection{Sketch of proof}
\label{sketch of proof}

Consider a pair of potential ancestral lineages, $\xi^{n,1}$ and $\xi^{n,2}$, 
created in some selective event which, without loss of generality, we suppose happens at time zero.
Suppose that we forget about further branches and when $\xi^{n,i}$ is affected by a neutral event it jumps to the location of the parent; when it
is affected by a selective event it jumps to the location of one of the potential parents 
(picked at random). Thus $\xi^{n,1}$ and $\xi^{n,2}$ are compound Poisson processes which 
interact when (and only when) $|\xi^{n,1}-\xi^{n,2}|\leq 2\mc{R}_n$.

We choose a large constant $c>0$.
We begin by showing that $\xi^{n,1}$ and $\xi^{n,2}$ have probability $\Theta(1/\log n)$ of reaching a distance $1/(\log n)^c$ from each other without coalescing (we then say they have diverged). 
We also show that the probability that $\xi^{n,1}$ and $\xi^{n,2}$ have not diverged or coalesced by time $1/(\log n)^c$ is $o(1/(\log n))$, so coalescence will be instantaneous in the limit.
Moreover, once they are $1/(\log n)^c$ 
apart, they won't get within distance $2\mc R_n$ of each other again on a timescale of $\mc{O}(1)$.
Hence from the point of view of our scaling they stay apart and evolve essentially independently of each other.

We exploit this observation by coupling the whole rescaled dual process with a process in which diverged lineages move independently. 
We use an object that we call a caterpillar which is defined in the same way as the rescaled dual process, except that selective events only result in branching if at least time $1/(\log n)^c$ has elapsed since the previous branching.
We stop the caterpillar at the first time a pair of lineages has either diverged or failed to coalesce in time $1/(\log n)^c$ after branching.
We then start two new independent caterpillars at the positions of the pair of lineages, and continue in the same way, giving a `branching caterpillar'.

The branching caterpillar can be coupled with the rescaled dual process by piecing together the independent Poisson point processes of events which drive each caterpillar into a single driving Poisson point process. 
We show that under the coupling, the branching caterpillar and the rescaled dual process coincide with high probability, using the result that lineages at a separation of at least $1/(\log n)^c$ are unlikely to interact again.
Each individual caterpillar converges in an appropriate sense to a segment of a Brownian path run for an exponentially distributed lifetime, so we can couple the branching caterpillar with the limiting branching Brownian motion.

This programme is carried out in Section~\ref{proof}.

\section{Biological background}
\label{biology}

In this section, we shall set our work in the context of the substantial
biological literature. The reader concerned only with the mathematics can safely skip to
Section~\ref{proof}.

The interplay between natural selection and the spatial structure of a population is a question
of longstanding interest in population genetics.
\cite{fisher:1937} studied the advance of selectively advantageous genetic types through
a one-dimensional population using the deterministic differential 
equation now known as the Fisher-KPP equation. 
This equation also makes sense in higher dimensions, but ignores {\em genetic drift}
(the randomness due to reproduction in a finite population). Work incorporating genetic
drift has been restricted to either one spatial dimension (see \cite{barton/etheridge/kelleher/veber:2013b} and
references therein) or, more commonly, to subdivided populations.
\cite{maruyama:1970} studied the probability of {\em fixation} of an advantageous genetic type
(the probability that eventually the whole population carries this genetic type) in 
a subdivided population. The assumptions made in that article are rather strong: if we think of the 
population as living on islands (or in colonies), then each island has constant total population size and its 
contribution to the next generation is in proportion to that size. Under these assumptions, the 
probability of fixation is not affected by the population structure: it is the same as for 
a gene of the same selective advantage in an unstructured population of the same total size.
Much subsequent work retained Maruyama's assumptions, and so it is often assumed that spatial structure has 
no influence on the accumulation of favourable genes. However,
\cite{barton:1993} showed that the extra stochasticity produced by the introduction of local extinctions and colonisations
could significantly change the fixation probability. This work was extended in, for example, 
\cite{cherry:2003} and \cite{whitlock:2003}. 

A fundamental problem in genetics is to identify which parts of
the genome have been the target of natural selection. The random nature
of reproduction in finite populations means that some
genetic types (alleles) will be carried by everyone in the population, 
even though they convey no particular selective advantage. However, if a favourable mutation
arises in a population and `sweeps' to fixation 
(i.e.~increases in frequency until everybody carries it), we
expect the genealogical trees (that is the trees of ancestral lineages) relating individuals in a sample from the population to differ
from those that we observe in the absence of selection. In particular, they will be more `star-shaped'.
Of course we cannot observe the genealogical trees directly, and so, instead, geneticists exploit the
fact that genes are arranged on chromosomes: the ancestry at another position on the same chromosome will be
correlated with that at the part of the genome that is the target of selection. 
In order to detect selection one therefore examines the patterns of 
variation at other points on the same chromosome, so-called linked loci.
 

In order for this approach to work, we require
sufficient variability at the linked loci that we see a signal of the distortion in the
genealogical tree.
This means that we must consider the genealogy of a sample from the population on the 
timescale set by the neutral mutation rate.  If selection is too strong, the genealogy will
be very short and we see no mutations and so we can recover no information about the 
genealogical trees; if selection is too weak, we won't be able to distinguish the
patterns from those seen under neutral evolution. 

Since neutral mutation rates are rather small, this means that we are interested in long timescales. 
Without selection, ancestral lineages in our model follow symmetric random walks with bounded variance jumps 
and so we expect a diffusive scaling to capture patterns of neutral variation. Since we are looking for deviations from those patterns due to the action of selection, it makes sense to consider a diffusive 
rescaling in the selective case too.  Thus, if the neutral mutation rate is $\mu$, then we look 
at the rescaled dual process with
$n=1/\mu$. If the branches produced by selection persist long enough to be visible at this scale, then
there is positive probability that the pattern of (neutral) variation we see in a sample from the population will look different from the pattern we'd expect without selection.

Our results in this paper are relevant to populations evolving in spatial continua. The question they 
address is `When can we hope to detect a signal of natural selection in data?'. Whereas in
the classical models of subdivided populations it is typically assumed that the population in each `island' is large, so that
neighbourhood size is big, by fixing the `impact' parameter $u$ in our model, we are assuming that
neighbourhood size is small. As a result, reproduction events are somewhat akin to local extinction and 
recolonisation events, in which a significant proportion of the local population is replaced in 
a single event. 
Our main result shows that our ability to detect selection is then critically
dependent on spatial dimension.  For populations living in at least three spatial 
dimensions (of which there are very few), spatial structure has a rather weak effect.  However, in
two spatial dimensions, selection must be stronger and in one spatial dimension (as appropriate
for example for populations living in intertidal zones) much stronger, before we can expect
to be able to detect it. The explanation is that in low dimensions, it is harder for
individuals carrying the favoured gene to escape the competition posed by close relatives who
carry the same gene. In our mathematical work, this is 
reflected in the vast majority of branches in our dual process being cancelled by a coalescence event on a timescale
which is negligible compared to the timescale set by the neutral mutation rate so that no evidence of these branches having occurred will be seen in the pattern of neutral mutation.

\section{Proof of Theorem~\ref{result d>1}}
\label{proof}

Our proof is broken into two steps. First in Subsection~\ref{excursionsec} we consider how the pair of potential ancestral lineages created during a selective event interact with each other.
In particular we find asymptotics for the probability that they diverge in a short time.
This will allow us to identify the branching
rate in the limiting Brownian motion. Then in Subsection~\ref{convtobbmsec}
we define the caterpillar and show how to couple the dual of the {\slfvs} to a system of branching caterpillars. With this construction in hand, Theorem \ref{result d>1} follows easily.

\subsection{Pairs of paths}
\label{excursionsec}

In this subsection we are interested in the behaviour of a pair of potential ancestral
lineages in the rescaled dual. In order that they be uniquely defined, if either is hit by a selective event then
we (arbitrarily) declare that it jumps to the location of the first potential parent sampled in 
that event. In particular, if they are both affected by the same event, then they will necessarily
coalesce.
We write $\xi^{n,1}$ and $\xi^{n,2}$ for the resulting potential ancestral lineages and 
$$\eta^n=\xi^{n,1}-\xi^{n,2}$$ 
for their separation.

Throughout this subsection, we use the notation $\P_{[r,r']}$ to mean that 
$|\eta^n_0|\in [r,r']$ and we adopt the convention that estimates of $\P_{[r,r']}[B]$ 
hold uniformly for all initial laws with mass concentrated on $[r,r']$. We extend this notation to open 
intervals in the obvious manner. 
We will also write $\P_r=\P_{[r,r]}$. 

We are concerned with the behaviour of two potential ancestral lineages created 
during a selective event which, without loss of generality, we suppose to happen at time $0$.
We shall then refer to $\eta^n$ as an excursion. 
In this case $|\eta^n_0|\leq 2\mc{R}_n$ and we wish to establish whether or not 
$|\eta_t^n|$ ever exceeds
\begin{equation}\label{gammadef}
\gamma_n=\frac{1}{(\log n)^{c}},
\end{equation}
where, in this section, we suppose that $c\geq 3$.

\begin{remark}
We will, eventually, set $c=4$, although any larger constant $c$ would give the same result;
for now we keep the dependence on $c$ visible in our estimates.
\end{remark}

For reasons that will soon become apparent, it is convenient to assume that $n$ is large enough that $7\mc{R}_n<\gamma_n$.

The picture of an excursion $\eta^n$ that we would like to build up is, loosely speaking, as follows. 
\begin{enumerate}
\item With probability $\kappa_n=\Theta(\frac{1}{\log n})$,
$|\eta^n|$ reaches displacement $\gamma_n$ within time $1/(\log n)^c$ and then $\xi^{n,1}$ and $\xi^{n,2}$ will not interact again before a fixed time $T>0$. Consequently the displacement between them becomes macroscopic and we see two distinct paths in the limit.  Moreover, $\kappa_n\log n \to \kappa\in(0,\infty)$ as $n\to\infty$.
\item With probability $1-\Theta(\frac{1}{\log n})$,
$|\eta^n|$ does not reach displacement $\gamma_n$, and $\xi^{n,1}$ and $\xi^{n,2}$ coalesce within time $1/(\log n)^c$. 
In this case the difference between them is microscopic and we see only one path in the limit. 
\item All other outcomes have probability $\mc O\big(\frac{1}{(\log n)^{c-3/2}}\big)$, which means that we won't see them
in the limit.
\end{enumerate}
Much of the work in making this rigorous results from the fact that $\xi^{n,1}$, $\xi^{n,2}$
only evolve independently when their separation is greater than $2\mc{R}_n$.  Our strategy is 
similar to that in the proof of Lemma~4.2 in \cite{etheridge/veber:2012}, but here we require
a stronger result: rather than an estimate of the form $\kappa_n\geq C/\log n$ we need
convergence of $ \kappa_n\log n$.

\subsubsection{Inner and outer excursions}
\label{inoutexc}

We shall characterise the behaviour of $\eta^n$ using several stopping times. Set $\tau^{out}_{0}=0$ and define inductively, for $i\geq 0$,
\begin{align}
\tau^{in}_i&=\inf\{s>\tau^{out}_i\-|\eta^n_s|\geq 5\mc{R}_n\},\label{intimes}\\
\tau^{out}_{i+1}&=\inf\{s>\tau^{in}_{i}\-|\eta^n_s|\leq 4\mc{R}_n\}.\notag
\end{align}
We refer to the interval $[\tau^{out}_{i},\tau^{in}_i)$ (and also to the path of $\eta^n$ during it) as the $i^{th}$ inner excursion and similarly to $[\tau^{in}_{i-1},\tau^{out}_i)$ (and corresponding path) as the $i^{th}$ outer excursion. 

Since a jump of $\eta^n$ has displacement at most $2\mc{R}_n$, although the initial ($0^{th}$)
inner excursion starts in $(0,2\mc{R}_n]$, for $i\geq 1$ we have $|\eta^n_{\tau^{in}_i}|\in [5\mc{R}_n,7\mc{R}_n]$ and $|\eta^n_{\tau^{out}_i}|\in[2\mc{R}_n,4\mc{R}_n]$. 

\begin{defn}\label{ioexctypedef}
We define the stopping times
\begin{align*}
\tau^{coal}&=\inf\{s>0\-|\eta^n_s|=0\},\\
\tau^{div}&=\inf\{s>0\-|\eta^n_s|\geq \gamma_n\},\\
\tau^{over}&=\frac{1}{(\log n)^c}.
\end{align*}
We shall say that the $i$th inner excursion coalesces if 
$\tau^{coal}\in [\tau^{out}_{i},\tau^{in}_i)$. Similarly, the $i$th outer excursion diverges if
$\tau^{div}\in [\tau^{in}_{i-1},\tau^{out}_i)$.

We define $\tau^{type}=\min(\tau^{coal},\tau^{div},\tau^{over})$ and say that $\eta^n$
\begin{enumerate}
\item coalesces if $\tau^{type}=\tau^{coal}$,
\item diverges if $\tau^{type}=\tau^{div}$,
\item overshoots if $\tau^{type}=\tau^{over}$.
\end{enumerate}
Since almost surely $\eta^n$ only jumps a finite number of times before time $(\log n)^{-c}$, almost surely $\tau^{type}$ occurs during either an inner or an outer excursion, whose index we
denote by $i^*$.
\end{defn} 

We use $\zeta^n$ to denote the distribution of the distance between the two 
potential parents sampled during a selective event. 

\begin{lemma}\label{istar}
There exists $\alpha\in(0,1)$ such that, uniformly in $n$, 
$\P_{\zeta^n}\l[i^*> m\r]\leq \alpha^m$.
\end{lemma}

\begin{lemma}\label{Pover}
As $n\to \infty$,
$\P_{\zeta^n}\l[\eta^n\text{ overshoots}\r]=\mc{O}\l(\frac{1}{(\log n)^{c- 3/2}}\r)$.
\end{lemma}

\begin{lemma}\label{Pdiv}
As $n\to \infty$,
$\P_{\zeta^n}\l[\eta^n\text{ diverges}\r]=\Theta\l(\frac{1}{\log n}\r)$.
\end{lemma}

\begin{lemma}\label{Pcoal}
As $n\to \infty$,
$\P_{\zeta^n}\l[\eta^n\text{ coalesces}\r]=1-\Theta\l(\frac{1}{\log n}\r)$.
\end{lemma}

Thus, overshoots are relatively unlikely, and typically $\eta^n$ consists of a finite number of inner/outer excursions until either (1) it coalesces, with probability $1-\Theta(\frac{1}{\log n})$,
or (2) the two lineages separate to distance $\gamma_n$, with probability $\Theta(\frac{1}{\log n})$. 

The remainder of this Section \ref{inoutexc} is devoted to the proof of Lemmas~\ref{istar}-\ref{Pdiv}. 
Lemma~\ref{Pcoal} then follows immediately, since $c\geq 3$.

We will need two more stopping times: 
\begin{align} \label{taurr}
\tau_r&=\inf\{s>0\-|\eta^n_s|\leq r\}, \notag \\
\tau^r&=\inf\{s>0\-|\eta^n_s|\geq r\}.
\end{align}
Note that $\tau_0=\tau^{coal}$. 

Note that the random variables $\tau^{type}$, $\tau^r$ and so on depend implicitly on $n$; throughout this section these random variables refer to the stopping times for the process $\eta^n$. 

\begin{proof}(Of Lemma \ref{istar}.) First consider a single inner excursion of $\eta^n$. It is easily seen that there exists some $\alpha'>0$ such that, for all $n$:
\begin{itemize}
\item[$(\dagger)$] For any $x\in(0,5\mc{R}_n)$, if $|\eta^n_t|=x$ then the probability that $\eta^n$ will hit $0$ but not exit $\mc{B}_{5\mc{R}_n}(0)$ within its next three jumps is at least $\alpha'$. 
\end{itemize}
In particular, the probability that the first three jumps of an inner excursion result in a coalescence is bounded away from $0$ uniformly for any $|\eta^n_{\tau^{out}_i}|\in [2\mc R_n, 4 \mc R_n]$. If $i^*>m$ then at least $m$ inner excursions must occur without a coalescence. The strong Markov property applied at the time $\tau^{out}_i$ means that, conditionally given $\eta^n_{\tau^{out}_i}$, the $i^{th}$ inner excursion is independent of $(\eta^n_t)_{t<\tau^{out}_i}$. Repeated application of this fact, coupled with $(\dagger)$, shows that the probability of seeing at least $m$ inner excursions without a single coalescence is at most $(1-\alpha')^{m}$. This completes the proof.
\end{proof}

We will shortly require a tail estimate on the supremum of the modulus of two dimensional Brownian motion $W$, which we record first for clarity. We write $W_t=(W^1_t,W^2_t)$ and note
\begin{align}
\P\l[\sup\limits_{s\in[0,t]}|W_s-W_0|\geq x\r]
&\leq 2\P\l[\sup\limits_{s\in[0,t]}|W^1_s-W^1_0|\geq x/2\r]\notag\\
&\leq 4\P\l[\sup\limits_{s\in[0,t]}(W^1_s-W^1_0)\geq x/2\r]\notag\\
&\leq 4e^{-x^2/8t}.\label{eq:2d_BM_sup}
\end{align}
In the first line of the above we use the triangle inequality and the fact that $W^1$ and $W^2$ have the same distribution. To deduce the second line, we note that $W^{1}$ and $-W^1$ have the same distribution. For the final line, we use the (standard) tail estimate $\P[\sup_{s\in[0,t]}(B_s-B_0)\geq x]\leq e^{-x^2/2t}$ for a one dimensional Brownian motion $B$, which can be deduced via Doob's martingale inequality applied to the submartingale $(\exp (xB_s /t))_{s\geq 0}$. 

During an outer excursion, $\eta^n$ is the difference between two independent walkers and so we
can use Skorohod embedding to approximate its behaviour using elementary calculations for 
two-dimensional Brownian motion. The next lemma exploits this to bound the duration of the
outer excursion and the probability that it diverges.
\begin{lemma}\label{exclengths}
As $n\to \infty$,
\begin{equation}
\P_{[5\mc{R}_n,7\mc{R}_n]}\l[\tau^{\gamma_n}\wedge\tau_{4\mc{R}_n}>(\log n)^{-c-1}\r]=\mc{O}\l(\frac{1}{(\log n)^{c-1}}\r),\label{Eouterbound}
\end{equation}
and
\begin{equation}
\P_{[5\mc{R}_n,7\mc{R}_n]}\l[\tau^{\gamma_n}<\tau_{4\mc{R}_n}\r]=\Theta\l(\frac{1}{\log n}\r).\label{excsucprob}
\end{equation}
\end{lemma}
\begin{proof}
For $i=1,2$ let $\hat{\xi}^{n,i}$ be a pair of independent processes such that $\hat{\xi}^{n,1}$ has the same distribution as $\xi^{n,1}$ and $\hat{\xi}^{n,2}$ has the same distribution as $\xi^{n,2}$. The process
$
\hat{\xi}^{n,1}-\hat{\xi}^{n,2}
$
is a compound Poisson process with a rotationally symmetric jump distribution and a maximum displacement of $2\mc{R}_n$ on each jump.  
Moreover (essentially by Skorohod's Embedding Theorem,
see e.g.~\cite{billingsley:1995}), we can construct a process $\hat{\eta}^n$ with the same distribution as $\hat{\xi}^{n,1}-\hat{\xi}^{n,2}$ as follows.

Let $(r_m,J_m)_{m\geq 1}$ denote a sequence distributed as the jump magnitudes and jump times of $\hat{\xi}^{n,1}-\hat{\xi}^{n,2}$. Let $W$ be a two-dimensional Brownian motion with $W_0=\hat{\xi}^{n,1}_0-\hat{\xi}^{n,2}_0$, independent of $(r_m,J_m)_{m\geq 1}$. Now set
\begin{align}
\label{skembed}
\hat{\eta}^n_t&=W_{T^{(S(t))}}\mbox{ where } T^{(0)}=0, J_0=0,\\
T^{(m+1)}&=\inf\{s>T^{(m)}\-|W_s-W_{T^{(m)}}|\geq r_m\},\notag\\
S(t)&=\sup\l\{i\geq 0\-J_i\leq t\r\}.\notag
\end{align}
We may then couple
$$\hat{\eta^n}=\hat{\xi}^{n,1}-\hat{\xi}^{n,2}.$$
We define $\hat{\tau}^r$ and $\hat{\tau}_r$ analogously to $\tau^r$ and $\tau_r$, as stopping times of the process $\hat{\eta}^n$. 

Note that since $(\xi^{n,1}_t,\xi^{n,2}_t)_{t\leq \tau_{4\mc{R}_n}}$ has the same distribution as $(\hat\xi^{n,1},\hat\xi^{n,2})_{t\leq \tau_{4\mc{R}_n}}$, we may couple them so that they are almost surely equal during this time. Thus 
$$\{\hat\tau^{\gamma_n}<\hat\tau_{4\mc{R}_n}\}=\{\tau^{\gamma_n}<\tau_{4\mc{R}_n}\}.$$

Let $T^r$ and $T_r$ be the analogues of $\tau^r$ and $\tau_r$ for $W$ (not to be confused with $T^{(m)}$ in \eqref{skembed}). By the definition of the Skorohod embedding in \eqref{skembed} we have 
\begin{align}
\P_{[5\mc{R}_n, 7\mc{R}_n]}\l[\hat\tau^{\gamma_n}<\hat\tau_{4\mc{R}_n}\r]
&\geq \P_{[5\mc{R}_n,7\mc{R}_n]}\l[T^{\gamma_n+2\mc{R}_n}<T_{4\mc{R}_n}\r]\notag\\
&\geq \P_{5\mc{R}_n}\l[T^{\gamma_n+2\mc{R}_n}<T_{4\mc{R}_n}\r].\label{skusenotimechange}
\end{align}
The right hand side concerns only the modulus of two-dimensional Brownian motion and so can be expressed in terms of the scale function for a two-dimensional Bessel process:
\begin{align}
\P_{5\mc{R}_n}\l[T^{\gamma_n+2\mc{R}_n}<T_{4\mc{R}_n}\r]=\frac{\log(5\mc{R}_n)-\log(4\mc{R}_n)}{\log(\gamma_n+2\mc{R}_n)-\log(4\mc{R}_n)}=\Theta\l(\frac{1}{\log n}\r),\label{exit1}
\end{align}
which proves the lower bound in \eqref{excsucprob}. Similarly, to see the upper bound we note that
\begin{align*}
\P_{[5\mc{R}_n,7\mc{R}_n]}\l[\hat\tau^{\gamma_n} < \hat\tau_{4\mc{R}_n}\r] &\leq \P_{[5\mc{R}_n,7\mc{R}_n]}[T^{\gamma_n}<T_{2\mc{R}_n}]\\
&\leq \P_{7\mc{R}_n}\l[T^{\gamma_n}<T_{2\mc{R}_n}\r]\\
&=\frac{\log(7\mc{R}_n)-\log(2\mc{R}_n)}{\log(\gamma_n)-\log(2\mc{R}_n)}\\
&=\Theta\l(\frac{1}{\log n}\r).
\end{align*}
It remains to prove \eqref{Eouterbound}. We have
$$\tau^{\gamma_n}\wedge \tau_{4\mc{R}_n}=\hat\tau^{\gamma_n}\wedge \hat\tau_{4\mc{R}_n}\leq \hat\tau^{\gamma_n}.$$
\begin{remark}
The above inequality is a very crude estimate, but will be enough to prove \eqref{Eouterbound}, which in turn will be enough to give useful bounds on the duration of excursions due to the freedom in the choice of $c$.
\end{remark}
Hence
\begin{equation}\label{eq:excursion_time}
\P_{[5\mc{R}_n,7\mc{R}_n]}\l[\tau^{\gamma_n}\wedge\tau_{4\mc{R}_n}>(\log n)^{-c-1}\r]\leq  
\P_{[5\mc{R}_n,7\mc{R}_n]}\l[|\hat{\eta}^n_{(\log n)^{-c-1}}|\leq \gamma_n \r].
\end{equation}
The remainder of the proof focuses on bounding the right side of \eqref{eq:excursion_time}. To do so, we must relate our compound Poisson process to another Brownian motion.

For $j\geq 1$, let $X_j=\hat{\eta}^n_{j/n}-\hat{\eta}^n_{(j-1)/n}$. 
Then $(X_j)_{j\geq 1}$ are i.i.d.~and since $\hat{\xi}^{n,1}$ and $\hat{\xi}^{n,2}$ are independent, 
$\E \l[|X_1|^2 \r]=2\E\l[|\hat{\xi}^{n,1}_{1/n}-\hat{\xi}^{n,1}_{0}|^2 \r]$.

Recall from \eqref{jump of size z} that the rate at which $\hat{\xi}^{n,1}$ jumps from $y$ to $y+z$ is determined by the intensity measure $m^n(dz)$
so that
\begin{equation}
\label{sigma squared}
\E \l[|X_1|^2 \r]= \frac{2}{n}\int_{\R^2} |z|^2 m^n (dz)
 = \frac{4 \sigma ^2}{n},
\end{equation}
where $\sigma^2$ was defined in~\eqref{first sigma squared}.
Now recall the definition of $S(t)$ in~\eqref{skembed}; the rate at which $\hat{\xi}^{n,1}$ jumps is $\int_{\R^2}m^n(z)dz=\Theta(n)$ by \eqref{jump rate}, so $S(n^{-1})$ is bounded by the sum of two Poisson$(\Theta(1))$ random variables. Hence since each jump of $\hat{\eta}^n$ is bounded by $2\mc R_n$,
\begin{align}
\E \l[ |X_1 |^4 \r]&\leq (2\mc R_n)^4 \E \l[ S(n^{-1})^4\r]=\mc O (n^{-2}).
\label{bound of fourth moment of X}
\end{align} 
Once again (since the distribution of $X_1$ is rotationally symmetric) we may use
Skorohod's Embedding Theorem to couple $(X_i)_{i\geq 1}$ to a two-dimensional Brownian motion $B$ started at $\eta^n_0$ and a sequence 
$\upsilon_1, \upsilon_2, \ldots $ of stopping times for $B$ such that setting $\upsilon_0=0$, $(\upsilon_i - \upsilon_{i-1})_{i\geq 1}$ are i.i.d. and 
\begin{align} \label{skembed2}
B_{\upsilon_i}-B_{\upsilon_{i-1}} &= X_i,\\
\E[\upsilon_i - \upsilon_{i-1}]&=\tfrac{1}{2}\E[|X_1|^2] =\tfrac{2 \sigma^2 }{n} \notag\\
\text{ and }\E[(\upsilon_i - \upsilon_{i-1})^2 ] &=  \mc{O}(n^{-2}). \notag
\end{align} 
It follows that $\E[\upsilon_{\lfloor tn \rfloor}]=\tfrac{2 \sigma^2\lfloor tn \rfloor}{n}$ and $\text{Var} (\upsilon_{\lfloor tn \rfloor})= \mc{O}(t n^{-1})$.
Hence by Chebychev's inequality,
$$\P[|\upsilon_{\lfloor tn \rfloor}-2 \sigma ^2 t|\geq n^{-1/3}]\leq  \mc{O}(t n^{-1/3}).  $$
Applying this result with $t=t_n:=(\log n)^{-c-1}$, since $\hat{\eta}^n_{\lfloor t_n n \rfloor /n}=B_{\upsilon_{\lfloor t_n n \rfloor}}$ we have
\begin{align}
&\P_{[5\mc{R}_n,7\mc{R}_n]}\l[|\hat{\eta}^n_{t_n}|\leq \gamma_n \r]\notag\\
&\hspace{2pc}\leq \P\Bigg[\inf\l\{|B_t-B_0|\-t\in[2\sigma ^2 t_n-n^{-1/3},2\sigma ^2 t_n+n^{-1/3}]\r\}\\
&\hspace{20pc}\leq \gamma_n +n^{-1/8}+7\mc R_n\Bigg]\notag\\
&\hspace{4pc} + \P \l[\big|\hat{\eta}^n_{t_n}- \hat{\eta}^n_{\lfloor t_n n\rfloor /n }\big|\geq n^{-1/8}\r] + \mc{O}(t_n n^{-1/3}).\label{eq:excursion_time_2}
\end{align}
For the first term on the right hand side we have for $n$ sufficiently large
\begin{align}
&\P\l[\inf\l\{|B_t-B_0|\-t\in[2\sigma ^2 t_n-n^{-1/3},2\sigma ^2 t_n+n^{-1/3}]\r\}\leq \gamma_n +n^{-1/8}+7\mc R_n\r]\notag\\
&\hspace{2pc}\leq\P \l[ |B_{2\sigma ^2 t_n}-B_0|\leq \gamma_n +3n^{-1/8} \r]+\P \l[ \sup_{t\in [0,2n^{-1/3}]} |B_t - B_0| \geq \tfrac{1}{2} n^{-1/8} \r]\notag\\
&\hspace{2pc}=\mc{O} (\gamma_n ^2  t_n^{-1} ) +\mc{O} (e^{-\frac{1}{64}n^{1/12}})\notag\\
&\hspace{2pc}=\mc O ((\log n)^{1-c}),\label{eq:excursion_time_3}
\end{align}
For the second inequality, we use that the density of $B_t$ is bounded by $(2\pi t)^{-1}$ for the first term and
we apply \eqref{eq:2d_BM_sup} for the second term. 

Moving on to the second term on the right hand side of \eqref{eq:excursion_time_2}, 
since from~(\ref{sigma squared}) we have
$\E \l[ |\hat{\eta}^n_{t_n}- \hat{\eta}^n_{\lfloor t_n n\rfloor /n }|^2 \r]=\mc O (n^{-1})$, by Markov's inequality 
\begin{equation}\label{eq:excursion_time_4}
\P \l[|\hat{\eta}^n_{t_n}- \hat{\eta}^n_{\lfloor t_n n\rfloor /n }|\geq n^{-1/8}\r]=\mc O (n^{-3/4}).
\end{equation}
Putting \eqref{eq:excursion_time_3} and \eqref{eq:excursion_time_4} into \eqref{eq:excursion_time_2} we have
$$\P_{[5\mc{R}_n,7\mc{R}_n]}\l[|\hat{\eta}^n_{t_n}|\leq \gamma_n \r]=\mc{O}((\log n)^{1-c}).$$
In view of \eqref{eq:excursion_time}, this completes the proof.
\end{proof}

\begin{proof}(Of Lemma \ref{Pover}.) First consider a single inner excursion. 
Evidently there exists $\beta>0$ such that, for all $n$:
\begin{itemize}
\item[$(\ddagger)$] For any $x\in(0,5\mc{R}_n)$, if $|\eta^n_t|=x$ then the probability that $\eta^n$ will either exit $\mc{B}_{5\mc{R}_n}(0)$ or hit $0$ within its next three jumps is at least $\beta$. 
\end{itemize}

Let $(J_l)_{l\geq 0}$ be the (a.s. finite) sequence of jump times of our inner excursion, and let $B_k$ be the event that the excursion either coalesces or exits $\mc{B}_{5\mc{R}_n}(0)$ at one of $\{J_{3k+1},J_{3k+2},J_{3k+3}\}$. By the strong Markov property (applied at $J_{3k}$) and $(\ddagger)$, $\inf\{k\geq 0\-\1 _{B_k}=1\}$ is stochastically bounded above by a geometric random variable $G$ with success probability $\beta $. 

Moreover, 
for as long as $\eta^n$ is not at $0$, the rate at which it jumps is bounded below by the rate at which $\xi^{n,1}$ jumps, which is $\int_{\R^2}m^n(dz)=\Theta(n)$ where $m^n$ is given by \eqref{jump of size z}. Hence for each $l\geq 0$, $J_{l+1}-J_l$ is stochastically bounded above by
$E_l$ where the $(E_i)_{i\geq 0}$ are i.i.d.~exponential random variables of this rate. 

Combining these observations, 
\begin{align}\label{innerlength}
\P_{(0,5\mc{R}_n)}\l[\tau^{5\mc{R}_n}\wedge \tau_0 > n^{-1/2}\r]&\leq 
\P \l[ J_{\lceil 3n^{1/3}+3\rceil}\geq n^{-1/2} \r]+\P\l[ G> n^{1/3}\r]\\ \nonumber
&=\mc{O}(n^{-1/6})+(1- \beta)^{n^{1/3}}=\mc O (n^{-1/6})
\end{align}
where the last line follows by Markov's inequality.

We are now in a position to complete the proof.
Recall that $\eta^n$ overshoots if 
it has neither coalesced nor diverged by time $(\log n)^{-c}$.
Let $n$ be sufficiently large that 
$$(\log n)^{1/2}(n^{-1/2}+(\log n)^{-c-1})\leq (\log n)^{-c}.$$
Thus, if $\eta^n$ overshoots and $i^*<(\log n)^{1/2}$, then at least one inner excursion must have lasted longer than $n^{-1/2}$ or at least one outer excursion must have lasted longer than $(\log n)^{-c}$. Hence,
\begin{align*}
\P_{\zeta^n}\l[\eta^n\text{ overshoots}\r]&\leq
(\log n)^{1/2}\bigg(\P_{(0,5\mc{R}_n)}\l[\tau^{5\mc{R}_n}\wedge \tau_0 > n^{-1/2}\r]\\
&\hspace{2.5cm}+\P_{[5\mc{R}_n,7\mc{R}_n]}\l[\tau^{\gamma_n}\wedge\tau_{4\mc{R}_n}>(\log n)^{-c-1}\r] \bigg)
\\
& \hspace{2cm} +\P_{\zeta^n}\l[i^*>(\log n)^{1/2}\r].
\end{align*}
Using \eqref{innerlength}, \eqref{Eouterbound} and Lemma \ref{istar} to bound the right hand side of the above equation, we obtain
\begin{align*}
\P_{\zeta^n}\l[\eta^n\text{ overshoots}\r]
&\leq (\log n)^{1/2} (\mc O (n^{-1/6})+\mc O ((\log n)^{1-c}))+\alpha^{(\log n)^{1/2}}\\
&=\mc{O}((\log n)^{3/2 -c}),
\end{align*}
which completes the proof.
\end{proof}

\begin{proof}(Of Lemma \ref{Pdiv}.) 
We note that the probability that $\eta^n$ diverges is bounded above by the probability that a divergent outer excursion occurs before a coalescing inner excursion occurs. Let us write $\eta^{n,i,in}$ for the $i^{th}$ inner excursion and $\eta^{n,i,out}$ for the $i^{th}$ outer excursion and let us write $\tau^{r,i,in},\tau_{r,i,in}$ and $\tau^{r,i,out},\tau_{r,i,out}$ for the associated equivalents of $\tau^r$ and $\tau_r$. Thus, 
\begin{align*}
&\P_{\zeta^n}\l[\eta^n\text{ diverges}\r] \\
&\hspace{1pc}\leq \P_{\zeta^n}\l[\inf\l\{i\geq 1\-\tau^{\gamma_n,i,out}<\tau_{4\mc{R}_n,i,out}\}
\leq \inf\{i\geq 0\-\tau_{0,i,in}<\tau^{5\mc{R}_n,i,in}\r\}\r].
\end{align*}
By the strong Markov property (applied successively at times $\tau^{out}_i$ and $\tau^{in}_i$), along with \eqref{excsucprob} and $(\dagger)$, the right hand side of the above equation is bounded above by the probability that a geometric random variable with success probability $\Theta(\frac{1}{\log n})$ is smaller than an (independent) geometric random variable with success probability $\alpha '>0$. With this in hand, an elementary calculation shows that
\begin{equation*}
\P_{\zeta^n}\l[\eta^n\text{ diverges}\r]=\mc{O}\l(\frac{1}{\log n}\r).
\end{equation*}
It remains to prove a lower bound of the same order.

In similar style to $(\dagger)$ and $(\ddagger)$, it is easily seen that there exists $\delta>0$ such that for all $n$:
\begin{itemize}
\item[$(\star)$] For any $x\in [\mc{R}_n,4\mc{R}_n]$, if $|\eta^n_0|=x$, the probability that $\eta^n$ will exit $\mc{B}_{5\mc{R}_n}(0)$ without coalescing is at least $\delta$.
\end{itemize}
We note also that $\zeta^n$ is equal to $n^{-1/2}\zeta^1$ in distribution, so since we assumed that $\mu((\tfrac{3}{4}\mc R, \mc R])>0$, there exists $\epsilon>0$ such that $\P[\zeta^n\geq \mc{R}_n]\geq \epsilon$ for all $n$. Thus, applying the strong Markov Property at time $\tau^{in}_0$ and using $(\star)$, we obtain
\begin{align*}
\P_{\zeta^n}\l[\eta^n\text{ diverges}\r]
&\geq \epsilon \delta \P _{[5 \mc R _n, 7\mc R_n]}\l[\tau^{\gamma_n}<\tau_{4\mc R_n} \r]-\P_{\zeta^n} \l[\eta^n \text{ overshoots} \r]\\
&=\Theta \l( \frac{1}{\log n}\r)
\end{align*}
as required, where the final statement follows from Lemma~\ref{exclengths} and Lemma~\ref{Pover} (since $c\geq 3$).
\end{proof}

\subsubsection{Production of branches}
\label{homogsec}

The next step of the proof of Theorem~\ref{result d>1} involves further analysis of 
pairs of potential ancestral lineages: first we need to check that once a pair has separated to a distance 
$\gamma_n$ they won't come back together again before a fixed time $K$; second we need to see that $\log n$ times the
divergence probability actually converges (c.f.~Lemma \ref{Pdiv}) as $n\rightarrow\infty$, since this will determine the
branching rate in our branching Brownian motion limit. These two statements are the object of the next two lemmas.

\begin{lemma}\label{Pinteract}
Fix $K\in(0,\infty)$. Then 
$$\P_{[(\log n)^{-c},\infty)}\l[\tau_{4 \mc R_n}\leq K\r]=\mc{O}\l(\frac{\log \log n}{\log n}\r).$$
\end{lemma}

\begin{lemma}\label{Pdiv2}
There exists $\kappa\in(0,\infty)$ such that $(\log n)\P_{\zeta^n}\l[\eta^n\text{ diverges}\r]\to\kappa$ as $n\to\infty$.
\end{lemma}

The remainder of this subsection is occupied with proving Lemmas \ref{Pinteract} and \ref{Pdiv2}.

\begin{proof}(Of Lemma \ref{Pinteract}.)
We use the Skorohod embedding of $\hat\eta$ 
into the Brownian motion $W$, as defined in \eqref{skembed}, to reduce the claim to an 
equivalent statement about a two-dimensional Bessel process.

Recall that $\eta^n_0=\hat\eta^n_0=W_0$ and recall 
$\tau_r$ from \eqref{taurr}, and that $\hat{\tau}_r$ and $T_r$ are the analogues of $\tau_r$ for $\hat{\eta}$ and $W$ respectively.
We have that $\eta^n_s=\hat\eta^n_s$ for all $s\leq \tau_{4\mc R_n }$ so 
\begin{align} \label{div_prob_time}
\P_{[(\log n)^{-c},\infty)}\l[\tau_{4 \mc R_n}\leq K\r]
&=\P_{[(\log n)^{-c},\infty)}\l[\hat\tau_{4 \mc R_n}\leq K\r] \nonumber \\ 
&\leq \P_{[(\log n)^{-c},\infty)}\l[T_{4 \mc R_n}\leq T^{(S(K))}\r],
\end{align}
where we used the Skorohod embedding given in \eqref{skembed} in the last line. 
For all $\tilde K,C>0$, since $T^{(k)}$ is increasing in $k$ we have
\begin{equation} \label{time_change_union}
\P\l[T^{(S(K))}\geq \tilde K \r]\leq \P \l[ S(K)\geq Cn \r] + \P \l[ T^{(Cn)} \geq \tilde K \r].
\end{equation} 
By its definition in \eqref{skembed}, $S(K)$ is bounded by the sum of two Poisson random variables with parameter $\chi=K\int_{\R^2}m^n (dz)$, where $m^n$ 
is given by~(\ref{jump of size z}). In particular, $\chi=\Theta(n)$.
Recall that if $Z'$ is Poisson with parameter $\chi$, then (using a Chernoff bound argument) for $k>\chi$,
\begin{equation}
\label{poisson tail}
\P[Z'>k]\leq \frac{e^{-\chi}(e\chi)^k}{k^k}.
\end{equation} 
Hence, for $C$ sufficiently large, there exists $\delta_1>0$ such that
\begin{equation} \label{delta_1_exp}
\P\l[S(K)\geq Cn \r]\leq \mc O (e^{-\delta_1 n}).
\end{equation}
Now by the definition of $(T^{(m)})_{m\geq 1}$ in \eqref{skembed}, and since $r_m\leq 2 \mc R_n$ for each $m$,
$$\P \l[ T^{(Cn)} \geq \tilde K \r] \leq \P\l[\sum_{i=1}^{Cn} R_i \geq \tilde K n \r],$$
where $(R_i)_{i\geq 1}$ is an i.i.d. sequence with $R_1 \stackrel{d}{=}\inf \{ t\geq 0 :|W_t|\geq 2\mc R\}.$ Since 
$$\P\l[ R_1\geq k\r]\leq \P \l[ R_1 \geq k-1 \r] \P \l[ |W_k -W_{k-1}|\leq 4\mc R\r]\leq 
\P \l[ |W_1 -W_{0}|\leq 4\mc R\r]^k,$$
there exists $\lambda>0$ such that $\E\l[ e^{\lambda R_1}\r] <\infty$. Hence by Cram\'er's theorem, for $\tilde K$ a sufficiently large constant, there exists $\delta_2>0$ such that 
\begin{equation} \label{delta_2_exp}
\P\l[ T^{(Cn)}\geq \tilde K\r] =\mc O (e^{-\delta_2 n}).
\end{equation}
By \eqref{div_prob_time} and \eqref{time_change_union} together with \eqref{delta_1_exp} and \eqref{delta_2_exp}, we now have for $\tilde K$ sufficiently large
\begin{equation} \label{time_change_div}
\P_{[(\log n)^{-c},\infty)}\l[\tau_{4 \mc R_n}\leq K\r]\leq \P_{[(\log n)^{-c},\infty)}
\l[T_{4 \mc R_n}\leq \tilde K\r]+\mc O (e^{-\delta_1 n})+\mc O (e^{-\delta_2 n}).
\end{equation}
To finish, we note that
\begin{align*}
&\P_{[(\log n)^{-c},\infty)}\l[ T_{4\mc R_n}\leq \tilde K \r] \\
&\hspace{2pc}\leq \sup\limits_{x\geq(\log n)^{-c}}\l(\P_{x} \l[ T_{4\mc R_n}\leq T^{x +\log n}\r]+\P_{x} \l[ T^{x+\log n}\leq \tilde K \r]\r)\\
&\hspace{2pc}\leq \sup\limits_{x\geq(\log n)^{-c}}\l(\frac{\log (x+\log n) -\log x}{\log(x+\log n)-\log (4\mc R_n)}\r)+\P\l[\sup_{t\leq \tilde K}|W_t -W_0|\geq \log n\r]\\
&\hspace{2pc}=\mc O \l( \frac{\log \log n}{\log n}\r)+\mc O (e^{-(8\tilde K)^{-1}(\log n)^2}),
\end{align*}
where the second line uses the scale function for a two-dimensional Bessel process, and the third line uses \eqref{eq:2d_BM_sup}.
Substituting this into \eqref{time_change_div}, we have the required result.
\end{proof}

\begin{proof} (Of Lemma \ref{Pdiv2}.)
Let $p_n:= \P_{\zeta^n} \l[ \tau^{\gamma_n}<\tau_0 \r]$. Note that by Lemma~\ref{Pover}, 
\begin{equation} \label{Pdiffoverdiv}
|p_n - \P_{\zeta^n} \l [ \eta^n \text{ diverges} \r] |=\mc O \l( \frac{1}{(\log n)^{c-3/2}}\r).
\end{equation}
Hence by Lemma \ref{Pdiv}, there exist $0<d\leq D<\infty$ such that for all $n\geq 2$,
$$
d\leq (\log n) p_n \leq D.
$$
It follows that $(p_n)_{n\geq 1}$ has a subsequence $(p_{n_k})_{k\geq 1}$ such that $(\log n_k ) p_{n_k}\to \kappa \in (0,\infty). $
Let $\epsilon>0$ and let $N\in \N$ be such that $N\geq 1/\epsilon$ and $|(\log N)p_N - \kappa | \leq \epsilon .$ 
By rescaling, noting that $\zeta^n\stackrel{d}{=}\zeta^N(\frac{N}{n})^{1/2}$, and similarly for $\eta^n$, we have
\begin{equation} \label{rescaled_pN}
p_N = \P_{\zeta^n}\l[\tau^{\gamma_N (Nn^{-1})^{1/2}}<\tau_0  \r]. 
\end{equation}
Recall, for clarity, that here (as throughout this section) $\tau^r$ and $\tau_0$ refer to the stopping times for the process $\eta^n$.

Define $X^{n,N} := |\eta^n_{\tau^{\gamma_N (N n^{-1})^{1/2}}}|$. Increasing $N$, we may assume that $7\mc R_n<\gamma_N (Nn^{-1})^{1/2}\leq \gamma_n$ for $n\geq N$. Thus,  
\begin{align} \label{SMPpn}
p_n &= \P_{\zeta^n}\l[\tau^{\gamma_N (Nn^{-1})^{1/2}}\leq \tau^{\gamma_n}<\tau_0\r] \notag \\
&= \E_{\zeta^n} \l[ \1 _{\tau^{\gamma_N (N n^{-1})^{1/2}}<\tau_0} \P _{X^{n,N}}\l[\tau^{\gamma_n}<\tau_0 \r] \r].
\end{align}
Here, the first line holds since $\zeta_n <\gamma_N (Nn^{-1})^{1/2}\leq \gamma_n$, and the second line follows from the first by applying the Strong Markov Property at time $\tau^{\gamma_N (N n^{-1})^{1/2}}$.

To estimate~(\ref{SMPpn}), note that 
$$X^{n,N} \in [l^{n,N},r^{n,N}]:=[\gamma_N(Nn^{-1})^{1/2},\gamma_N(Nn^{-1})^{1/2}+2\mc R_n].$$ 
Using the Skorohod embedding defined in \eqref{skembed},
\begin{align} \label{pnpNlower}
\P_{[l^{n,N},r^{n,N}]} \l[ \tau^{\gamma_n} <\tau_0\r]& \geq \inf_{x\geq \gamma_N(Nn^{-1})^{1/2}}\P_x \l[ \tau^{\gamma_n} <\tau_{7\mc R_n}\r]\nonumber \\
&\geq \inf_{x\geq \gamma_N(Nn^{-1})^{1/2}}\P_x \l[ T^{\gamma_n + 2\mc R_n}<T_{7\mc R_n}\r] \nonumber \\
&= \frac{\log (\gamma_N (Nn^{-1})^{1/2})-\log (7\mc R_n)}{\log (\gamma_n +2\mc R_n)-\log (7\mc R_n)} \nonumber \\
&= \frac{\frac{1}{2} \log N +\mc O (\log \log N)}{\frac{1}{2} \log n +\mc O (\log \log n)}.
\end{align}
Note that, in the above, we (again) use the scale function for a two-dimensional Bessel process to deduce the third line.

We require slightly more work to establish an upper bound. We have
\begin{equation} \label{pnpNupper_1}
\P_{[l^{n,N},r^{n,N}]} \l[ \tau^{\gamma_n} < \tau_0\r] \leq \P_{[l^{n,N},r^{n,N}]} \l[ \tau^{\gamma_n}<\tau_{7 \mc R_n }\r] + \P_{[l^{n,N},r^{n,N}]} \l[ \tau_{7\mc R_n} <\tau^{\gamma_n} <\tau_0 \r].
\end{equation}
We begin by controlling the second term on the right hand side of \eqref{pnpNupper_1}.
By the Strong Markov Property at time $\tau_{7 \mc R_n}$,
\begin{align*}
\P_{[l^{n,N},r^{n,N}]} \l[ \tau_{7\mc R_n}<\tau^{\gamma_n} <\tau_0 \r]&= \E_{[l^{n,N},r^{n,N}]} \l[\1 _{\tau_{7\mc R_n}<\tau^{\gamma_n}} \P_{|\eta^n_{\tau_{7\mc R_n}}|} \l[ \tau^{\gamma_n} <\tau_0 \r] \r]\\
&\leq \E_{[l^{n,N},r^{n,N}]} \l[ \P_{|\eta^n_{\tau_{7\mc R_n}}|} \l[ \tau^{\gamma_n} <\tau_0 \r] \r]. 
\end{align*}
Since $\big|\eta^n_{\tau_{7\mc R_n}}\big|\in [5 \mc R_n, 7 \mc R_n]$, using \eqref{excsucprob} in the same way as in the proof of Lemma \ref{Pdiv},
\begin{equation} \label{upper_logn}
\P_{[l^{n,N},r^{n,N}]} \l[ \tau_{7\mc R_n}<\tau^{\gamma_n} <\tau_0 \r]=\mc O \l(\frac{1}{\log n}\r). 
\end{equation}
Next, we control the first term on the right hand side of \eqref{pnpNupper_1}, again using 
the Skorohod embedding \eqref{skembed}: 
\begin{align} \label{pnpNupper_2}
\P_{[l^{n,N},r^{n,N}]} \l[ \tau^{\gamma_n} <\tau_{7 \mc R_n}\r]& \leq \P_{[l^{n,N},r^{n,N}]} \l[ T^{\gamma_n}<T_{5\mc R_n}\r] \nonumber \\
&\leq \frac{\log (\gamma_N (Nn^{-1})^{1/2}+2 \mc R_n)-\log (5\mc R_n)}{\log (\gamma_n )-\log (5\mc R_n)} \nonumber \\
&= \frac{\frac{1}{2} \log N +\mc O (\log \log N)}{\frac{1}{2} \log n +\mc O (\log \log n)}.
\end{align}
Combining \eqref{pnpNlower}, \eqref{pnpNupper_1}, \eqref{upper_logn} and \eqref{pnpNupper_2}, 
$$\P_{[l^{n,N},r^{n,N}]} \l[ \tau^{\gamma_n} < \tau_0\r]=\frac{\log N +\mc O (\log \log N)}{ \log n +\mc O (\log \log n)}+\mc O \l(\frac{1}{\log n} \r). $$
Hence by \eqref{SMPpn},
\begin{align*}
p_n&=\P_{\zeta^n} \l[ \tau^{\gamma_N (N n^{-1})^{1/2}}<\tau_0 \r]\l(\frac{\log N +\mc O (\log \log N)}{ \log n +\mc O (\log \log n)}+\mc O \l(\frac{1}{\log n} \r)\r)\\
&=\frac{(\log N)p_N}{\log n}\l(\frac{1 +\mc O \big(\frac{\log \log N}{\log N}\big)}{ 1 +\mc O \big(\tfrac{\log \log n}{\log n}\big)}+\mc O \l(\frac{1}{\log N} \r)\r),
\end{align*}
where we used \eqref{rescaled_pN} in the last line.
Since $ |(\log N)p_N - \kappa | \leq \epsilon  $ we obtain for $n\geq N$
\begin{align*}
(\log n)p_n &\geq (\kappa -\epsilon )\l(\frac{1 +\mc O (\frac{\log \log N}{\log N})}{ 1 +\mc O (\frac{\log \log n}{\log n})}+\mc O \l(\frac{1}{\log N} \r)\r)\\
\text{ and }\quad (\log n)p_n &\leq (\kappa +\epsilon )\l(\frac{1 +\mc O \big(\frac{\log \log N}{\log N}\big)}{ 1 +\mc O \big(\frac{\log \log n}{\log n}\big)}+\mc O \l(\frac{1}{\log N} \r)\r).
\end{align*} 
Letting $\epsilon \to 0$ and hence $N\to \infty$, $\lim_{n \to \infty} (\log n)p_n = \kappa$. The result follows by \eqref{Pdiffoverdiv}.
\end{proof}

\subsection{Convergence to branching Brownian motion}
\label{convtobbmsec}

In this subsection we identify particular subsets of the dual process that we couple with objects that we
call `caterpillars'. The caterpillars play the r\^ole of individual branches in the limiting branching Brownian
motion. Our (eventual) goal is to write down a system of `branching caterpillars' and couple it to the {\slfvs} dual.
Establishing these couplings is greatly simplified by viewing the branching and coalescing 
dual as a deterministic function of an augmented driving Poisson point process and so our
first task is to recast the {\slfvs} dual in this way. 

Recall that we have a fixed impact parameter $u\in (0,1]$.  We define, recursively,
a sequence of subsets of $[0,1]$ as follows:
$$A^1_u=[0,u],\mbox{ and for }k\geq 1, \, A_u^{k+1}=uA^k_u\cup (u+(1-u)A_u^k).$$
Then if $U\sim \text{Unif}[0,1]$, $(\1_{A^k_u}(U))_{k\geq 1}$ is an i.i.d.~sequence of Bernoulli$(u)$ random variables
(see Lemma 3.20 in \cite{kallenberg:2006} for a proof in the case $u=\frac{1}{2}$, where $(\1_{A^k_u}(U))_{k\geq 1}$ is the binary expansion of $U$; the general case is an easy extension of this).

Let 
$$\mathscr{X}=\R \times \R^2 \times \R_{+} \times \mathcal B_1(0)^2\times [0,1]^2.$$

\begin{defn}[The dual as a deterministic function of a driving point process] \label{slfvs_dual_determ}

Given a simple point process $\Pi$ on $\mathscr{X}$, and some $p\in \R^2$, we define $(\mathcal P _t (p,\Pi))_{t\geq 0}$ as a process on $\cup_{k=1}^\infty (\R^2)^k$ as follows. 

For each $t\geq 0$, $\mathcal P _t (p,\Pi)=(\xi _t ^1,\ldots ,\xi^{N_t}_t)$ for some $N_t \geq 1$. We refer to $i$ as the index of the ancestor $\xi^i_t$. We begin at time $t=0$ from a single ancestor $\mathcal P_0(p,\Pi)=\xi_0^1=p$ and proceed as follows.

At each $(t,x,r,z_1,z_2,q,v)\in \Pi$ with $v\geq \v {s_n}$, a neutral event occurs:
\begin{enumerate}
\item Let $\xi ^{n_1}_{t-},\ldots , \xi ^{n_m}_{t-}$ denote the ancestors in $\mathcal B _r(x)$ which have not yet coalesced with an ancestor of lower index, with $n_1<\ldots < n_m$. For $1\leq i\leq m$, mark the ancestor $\xi ^{n_i}_{t-}$ iff $q\in A_u^i$. Let $\xi ^{r_1}_{t-},\ldots \xi ^{r_l}_{t-}$ denote the marked ancestors.
\item If at least one ancestor is marked, we set $\xi ^{r_i}_{t}=x+rz_1$ for each $i$ and call this the parental location for the event. We say that the ancestor $\xi ^{r_i}_t$ has coalesced with the ancestor $\xi^{r_1}_t$, for each $i\geq 2$.
\end{enumerate}
At each $(t,x,r,z_1,z_2,q,v)\in \Pi$ with $v<\v {s_n}$, a selective event occurs:
\begin{enumerate}
\item Let $\xi ^{n_1}_{t-},\ldots , \xi ^{n_m}_{t-}$ denote the ancestors in $\mathcal B _r(x)$ which have not yet coalesced with an ancestor of lower index, with $n_1<\ldots < n_m$. For $1\leq i\leq m$, mark the ancestor $\xi ^{n_i}_{t-}$ iff $q\in A_u^i$. Let $\xi ^{r_1}_{t-},\ldots \xi ^{r_l}_{t-}$ denote the marked ancestors.
\item If at least one ancestor is marked, we set $\xi ^{r_i}_{t}=x+rz_1$ for each $i$ and add an ancestor $\xi ^{N_{t-}+1}_{t}=x+rz_2$. We call $x+rz_1$ and $x+rz_2$ the parental locations of the event. We say that the ancestor $\xi ^{r_i}_t$ has coalesced with the ancestor $\xi^{r_1}_t$, for each $i\geq 2$.
\end{enumerate}
For each $l\in \N$, if $\xi^l_{\tau}$ has coalesced with an ancestor $\xi^k_{\tau}$ of lower index at time $\tau$, we set $\xi^l_t = \xi_t^{k}$ for all $t\geq \tau$.
\end{defn}
In the same way as for the definition of $\mc P^{(n)}(p)$ before the statement of Theorem \ref{result d>1}, we shall view $(\mathcal P _t (p,\Pi))_{t\geq 0}$ as a collection of potential ancestral lineages.
Given a realization of $\Pi$, we say that a path that begins at $p$ is a potential ancestral lineage if (1) at each neutral event that it encounters, it moves to the (single) parent and (2) at each selective event it encounters, it moves to one of the parents of that event. 

Note that if $\Pi$ is a Poisson point process on $\mathscr{X}$ with intensity measure 
\begin{equation}\label{eq:ppp_intensity}
n\,dt\otimes n\,dx 
\otimes \mu^n(dr)\otimes \pi^{-1} dz_1 \otimes \pi^{-1} dz_2 \otimes dq \otimes dv
\end{equation}
then as a collection of potential ancestral lineages, $(\mathcal P _t (p,\Pi))_{t\geq 0}$ has the same distribution as $\mc P^{(n)}(p)$.

When $\Pi$ takes this form, the result is that the driving Poisson Point Process in \eqref{rescalingeq} has been augmented by components that determine the nature of each event (neutral or selective), the parental locations of each event and which lineages in the region of the event are affected by it.
We have abused notation by retaining the notation $\Pi$ for this augmented process.

\subsubsection{The caterpillar}

We now introduce the notion of a caterpillar, which involves following a pair of potential ancestral lineages in the dual. 
We stop the caterpillar if the pair of lineages reaches displacement of $(\log n)^{-c}$, or if the pair does not coalesce within time $(\log n)^{-c}$ after last branching. While doing so, we suppress the creation of the second potential parent at any selective events that occur within time $(\log n)^{-c}$ of the previous (unsuppressed) selective event.

Let $\Pi$ be a Poisson point process on $\mathscr{X}$ with intensity measure \eqref{eq:ppp_intensity}. We write $(\mc P_t(p,\Pi))_{t\geq 0}=(\xi _t ^1,\ldots ,\xi^{N_t}_t)_{t \geq 0}$ as defined in Definition \ref{slfvs_dual_determ}. 

\begin{defn}[Caterpillar] \label{caterpillar_defn}
For $p\in \R^2$, we define a lifetime $h(p,\Pi)>0$, and a process $(c_t (p,\Pi))_{0 \leq t \leq h(p,\Pi)}$ on $(\R^2)^2$, which we shall refer to as a caterpillar. For each $t\geq 0$, we write 
$$c _t (p,\Pi)=\l(c_t ^1(p,\Pi),c^{2}_t(p,\Pi)\r),$$
dropping the dependence on $(p,\Pi)$ from our notation, when convenient. As part of the definition, we will also define $k^*(p,\Pi)\in \N$ and a sequence $(\tau^{\text{br}}_k)_{k\leq k^*}$ of stopping times.

Set $\tau^{\text{br}}_0=0$ and let $\tau^{\text{br}}_{1}$ be the time of the first selective event after $(\log n)^{-c}$ to affect $\xi^1$.
For $t \leq \tau^{\text{br}}_{1}$, let $c_t^1=c_t^2=\xi^1_t$.

Then, for $k\geq 1$, suppose we have defined $(\tau^{\text{br}}_l)_{l \leq k}$; let $m(k)=N_{\tau^{\text{br}}_k}$. 

For $t\in[\tau^{\text{br}}_k, \tau^{\text{br}}_k+(\log n)^{-c}]$, define $c_t^1(p,\Pi)=\xi^1_t$ and $c_t^2(p,\Pi)=\xi^{m(k)}_t$. 

In analogy with Definition~\ref{ioexctypedef},
define 
\begin{align} \label{tau_type_cat}
\tau_k^{\text{div}}&=\inf \{t\geq \tau^{\text{br}}_k :|c^1_t -c^2_t| \geq (\log n)^{-c}\}, \notag\\
\tau_k^{\text{coal}}&=\inf \{t\geq \tau^{\text{br}}_k :c^1_t=c^2_t\}, \notag \\
\tau_k^{\text{over}}&=\tau^{\text{br}}_k +(\log n)^{-c},
\end{align}
and let $\tau_k^{\text{type}}=\min(\tau_k^{\text{div}}, \tau_k^{\text{coal}},\tau_k^{\text{over}})$.  If $\tau_k^{\text{type}}\neq \tau_k^{\text{coal}}$ then set $k^*(p,\Pi)=k$ and $h(p,\Pi)=\tau_{k^*}^{\text{type}}$. The definition is then complete. If not, we proceed as follows.

Let $\tau^{\text{br}}_{k+1}$ be the time of the first selective event occurring strictly after $\tau^{\text{br}}_k+(\log n)^{-c}$ to affect $\xi^1$. For $t\in[\tau^{\text{br}}_k+(\log n)^{-c}, \tau^{\text{br}}_{k+1})$, let $c_t^1(p,\Pi)=c_t^2(p,\Pi)=\xi^1_t$. 

We then continue iteratively for each $k\leq k^*(p,\Pi)$.
\end{defn}
We refer to $(\tau^{\text{br}}_k)_{k\leq k^*}$, the times at which a selective event results in branching, as branching events.
We shall abuse our previous terminology and say that a branching
event diverges, coalesces or overshoots when the same is true of the excursion corresponding to 
the pair $(c^1, c^2)$.

\begin{remark} Note that $(c_t)_{t \geq 0}$ is not a Markov process with respect to its natural filtration, since $c^1$ and $c^2$ are not allowed to branch off from each other within $(\log n)^{-c}$ of the previous branching event. However, for $i=1,2$, $(c^i_t (p,\Pi))_{0 \leq t \leq h(p,\Pi)}$ is a Markov process with the same jump rate and jump distribution as a single potential ancestral lineage in the rescaled {\slfvs} dual. Moreover for each $1\leq k \leq k^*$, $(c^1_t, c^2_t)_{\tau^{\text{br}}_k \leq t \leq \tau^{\text{type}}_k}$ is an excursion as defined in Section \ref{excursionsec}.
\end{remark}

Recall the definition of $m^n(dz)$ from \eqref{jump of size z} and let
\begin{equation} \label{kappa_n_lambda}
\kappa_n = (\log n)\P [ \tau_1^{type} \neq \tau_1^{coal}]\quad \text{ and }\quad \lambda = n^{-1} \int_{\R^2}m^n(dz)=\Theta (1). 
\end{equation}
By combining Lemma \ref{Pdiv2} and Lemma \ref{Pover},
\begin{equation}\label{kappaconv}
\kappa_n \to \kappa
\end{equation}
as $n\to \infty$.

By the strong Markov property of $\Pi$, and since $\tau_k^{\text{type}}\leq \tau^{\text{br}}_{k}+(\log n)^{-c}\leq \tau^{\text{br}}_{k+1}$ 
for each $k$, the types of the selective events, $(\{\tau_k^{\text{type}}=\tau_k^{\text{div}}\})_{k\geq 1}$, $(\{\tau_k^{\text{type}}=\tau_k^{\text{coal}}\})_{k\geq 1}$ and $(\{\tau_k^{\text{type}}=\tau_k^{\text{over}}\})_{k\geq 1}$  are each i.i.d. sequences.
Thus,
\begin{equation} \label{k_star_distn}
k^*(p,\Pi)\sim \text{Geom}(\kappa _n (\log n)^{-1}).
\end{equation}
By \eqref{kappaconv}, there exist constants $0<a\leq A <\infty$ such that $\kappa_n \in [a,A]$ for all $n$ sufficiently large, so
\begin{equation} \label{k_star_exp_bound}
\P[k^*\geq (\log n)^{9/8}]=(1-\tfrac{\kappa_n}{\log n})^{(\log n)^{9/8}}=\mc O(e^{-\delta (\log n)^{1/8}})
\end{equation}
for some $\delta>0$.

\begin{lemma} \label{lifetime}
We can couple $h(p,\Pi)$ with $H\sim \text{Exp}(\kappa_n \lambda)$ in such a way that for some $\delta>0$, with probability at least $1-\mc O(e^{-\delta (\log n)^{1/8}})$
$$|h(p,\Pi)-H|\leq 3(\log n)^{-1/4}.$$
\end{lemma}

\begin{proof}
Recall the definition of $\lambda$ in~\eqref{kappa_n_lambda}.
Since the total rate at which $c^1$ jumps is given by $\lambda n$, and each jump is from a selective event independently with probability $\v{s_n}=\frac{\log n}{n}$, by the strong Markov property of $\Pi$ we have that 
\begin{equation} \label{eq:Ek_distn}
E_k:=\tau^{\text{br}}_{k}-(\tau^{\text{br}}_{k-1}+(\log n)^{-c})\sim \text{Exp}(\lambda \log n)
\end{equation}
and $(E_k , \1_{\{\tau_k^{\text{type}}\neq \tau_k^{\text{coal}}\}})_{k\geq 1}$ is an i.i.d. sequence.

Since (for example) $\{\tau_k^{\text{type}}\neq \tau_k^{\text{coal}}\}$ is not independent of the radius of the event at $\tau^{\text{br}}_k$, we note that $E_k$ and $\1_{\{\tau_k^{\text{type}}\neq\tau_k^{\text{coal}}\}}$ are not independent; therefore $(E_k)_{k\geq 1}$ is not independent of $k^*$. However, we can couple $(E_k , \1_{\{\tau_k^{\text{type}}\neq\tau_k^{\text{coal}}\}})_{k\geq 1}$ with a sequence $(E'_k)_{k\geq 1}$ which is independent of $k^*$ as follows. 

First sample the sequence $(\1_{\{\tau_k^{\text{type}}\neq\tau_k^{\text{coal}}\}})_{k\geq 1}$, and then independently sample a sequence $(E'_k , A_k)_{k\geq 1}$ with the same distribution as $(E_k , \1_{\{\tau_k^{\text{type}}\neq\tau_k^{\text{coal}}\}})_{k\geq 1}$. 
Then, for each $k\geq 1$, if $A_k=\1_{\{\tau_k^{\text{type}}\neq\tau_k^{\text{coal}}\}}$ set $E_k = E'_k$, and if not sample $E_k$ according to its conditional distribution given $\1_{\{\tau_k^{\text{type}}\neq\tau_k^{\text{coal}}\}}$. 

We now have a coupling of $(E_k , \1_{\{\tau_k^{\text{type}}\neq\tau_k^{\text{coal}}\}})_{k\geq 1}$ and $(E'_k)_{k\geq 1}$ such that $(E'_k)_{k\geq 1}$ is an i.i.d.~sequence, independent of $k^*$, with $E'_1 \sim \text{Exp}(\lambda \log n)$. Also, since $\P[\tau_k^{\text{type}}\neq \tau_k^{\text{coal}}]=\Theta((\log n)^{-1})$, we have that independently for each $k$, $E_k=E'_k$ with probability at least $1-\Theta((\log n)^{-1})$. 

We write
$$\sum_{k=1}^{k^*}E_k = \sum_{k=1}^{k^*}E'_k + \sum_{k=1}^{k^*}D_k, $$
where $D_k=E_k-E'_k$ and, by \eqref{k_star_distn}, $\sum_{k=1}^{k^*}E'_k \sim \text{Exp}(\lambda \kappa_n)$. 

Our next step is to bound $\sum_{k=1}^{k^*}D_k$. Firstly, applying a Chernoff bound to the binomial distribution yields
\begin{align} 
&\P\l[\l|\big\{k<(\log n)^{9/8} :D_k\neq 0\big\}\r|\geq (\log n)^{1/4}\r] \notag\\
&\hspace{4pc} = \P\l[ \text{Bin}\big((\log n)^{9/8}, \Theta ((\log n)^{-1})\big)\geq (\log n)^{1/4}\r] \notag\\
&\hspace{4pc} = \mc O \big(\exp (-\delta ' (\log n)^{1/4})\big)\label{chernoff_zk}
\end{align}
for some $\delta'>0$. Secondly,
\begin{align} \label{bound_zk}
\P \l[ |D_1 | \geq (\log n)^{-1/2} \r]&\leq \P \l[ E_1  \geq \tfrac{1}{2}(\log n)^{-1/2} \r]+\P \l[ E'_1 \geq \tfrac{1}{2}(\log n)^{-1/2} \r] \notag \\
&= 2 \exp(-\lambda (\log n)^{1/2}/2).
\end{align}
Combining \eqref{k_star_exp_bound}, \eqref{chernoff_zk} and \eqref{bound_zk}, we have that \begin{equation} \label{bound_zk_sum}
\P\l[\sum_{k=1}^{k^*}D_k\geq (\log n)^{-1/4}\r]=\mc  O \l(e^{-\delta '' (\log n)^{1/8}}\r),
\end{equation}
for some $\delta''\in (0,\delta)$.

Note that 
\begin{equation*} 
\sum_{k=1}^{k^*}E_k=\tau^{\text{br}}_{k^*}-k^*(\log n)^{-c}=h-k^*(\log n)^{-c}-(\tau_{k^*}^{\text{type}}-\tau_{k^*}^{\text{br}}),
\end{equation*}
with $0\leq \tau_{k^*}^{\text{type}}-\tau_{k^*}^{\text{br}}\leq (\log n)^{-c}$.
Let $H=\sum_{k=1}^{k^*}E'_k$. Then by \eqref{k_star_exp_bound} and \eqref{bound_zk_sum}, we have
\begin{equation*} 
\P\l[|h(p,\Pi)-H|\geq (\log n)^{9/8-c}+(\log n)^{-c}+(\log n)^{-1/4}\r]=\mc O\l(e^{-\delta '' (\log n)^{1/8}}\r).
\end{equation*}
The result follows since $c\geq 3$.
\end{proof}

Our next step is to show that a caterpillar is unlikely to end with an overshooting event.
\begin{lemma} \label{caterpillar_overshoot}
As $n \to \infty$,
$\P \l[\tau_{k^*}^{\text{type}}= \tau_{k^*}^{\text{over}}\r]=\mc O \l((\log n)^{\frac{21}{8}-c}\r).$
\end{lemma}
\begin{proof}
By Lemma \ref{Pover}, for $k\geq 1$
\begin{equation} \label{prob_undec}
\P[\tau_k^{\text{type}}=\tau_k^{\text{over}}]=\mc O((\log n)^{\frac{3}{2}-c}).
\end{equation}
Moreover,  
$$ \{\tau_{k^*}^{\text{type}}=\tau_{k^*}^{\text{over}}\}
\subset \{k^*\geq (\log n)^{9/8}\}\cup 
\bigcup _{k=1}^{(\log n)^{9/8}}\{\tau_k^{\text{type}}=\tau_k^{\text{over}}\}.$$
It follows, using \eqref{k_star_exp_bound}, that
\begin{align*}
\P[\tau_{k^*}^{\text{type}}=\tau_{k^*}^{\text{over}}]
&=\mc O(e^{-\delta(\log n)^{1/8}})+\mc O((\log n)^{\frac{3}{2} +\frac{9}{8}-c})=\mc O((\log n)^{\frac{21}{8}-c}).
\end{align*}
This completes the proof.
\end{proof}

We now show that a single caterpillar can be coupled to a Brownian motion in such a way that the caterpillar closely follows the Brownian motion, during time $[0,h(p,\Pi)]$.

Recall that the rate at which $\xi^1$ jumps from $y$ to $y+z$ is given by intensity measure
$m^n(dz)$, defined in \eqref{jump of size z}. 
Thus for $(c_t)_{t\geq 0}$ started at $p$, $\E[c^1_t]=p$ and the covariance matrix of $c^1_t$ is $ \sigma ^2 t \text{Id}$ since by \eqref{first sigma squared}, 
$$\sigma ^2 = \tfrac{1}{2}\int _{\R^2}|z|^2 m^n (dz). $$

Armed with this, the following lemma is no surprise.
\begin{lemma} \label{single_caterpillar}
Let $(W_t)_{t\geq 0}$ 
be a two-dimensional Brownian motion with $W_0=p$.
We can couple $(c _t (p,\Pi))_{t\leq h(p,\Pi)}$ with $(W_t)_{t\geq 0}$, in such a way that 
$(W_t)_{t\geq 0}$ is independent of $(\tau^{br}_k)_{k\geq 1}$ and $k^*(p,\Pi)$, and for any $r>0$, with probability at least $1-\mc O((\log n)^{-r})$,
for $t\leq h(p,\Pi)$, 
$$|c^1_t(p,\Pi)-W_{\sigma^2 t}|\leq (\log n)^{\frac{9}{8}-\frac{c}{3}}.$$
\end{lemma}
\begin{remark}
By the definition of the caterpillar in Definition \ref{caterpillar_defn}, for all $t\leq h(p,\Pi)$, $|c^2_t-c^1_t|\leq (\log n)^{-c}$. Hence under the coupling of Lemma \ref{single_caterpillar}, with probability at least $1-\mc O((\log n)^{-r})$, 
$|c^2_t(p,\Pi)-W_{\sigma^2 t}|\leq 2(\log n)^{\frac{9}{8}-\frac{c}{3}}$. 
\end{remark}
\begin{proof}
The proof is closely related to the second half of the proof of Lemma~\ref{exclengths}.
Note for $k\geq 0$, on the time interval $[\tau^{br}_k+(\log n)^{-c},\tau^{br}_{k+1})$, $c^1_t$ is a pure jump process with rate of jumps from $y$ to $y+z$ given by $(1-\v{s_n})m^n(dz)$.
Let $(\tilde c _t)_{t \geq 0}$ be a pure jump process with $\tilde c_0=0$ and rate of jumps from $y$ to $y+z$ given by $(1-\v{s_n})m^n(dz)$.
For $i\geq 1$, let 
$$X_i = \tilde c _{i/n} - \tilde c _{(i-1)/n}.$$
Then $(X_i)_{i\geq 1}$ are i.i.d., and as in \eqref{sigma squared} and \eqref{bound of fourth moment of X}, we have $\E[|X_1|^2]=\frac{2\sigma^2 (1-\v{s_n}) }{n}$ and $\E[|X_1|^4]= \mc O(n^{-2})$.

By the same Skorohod embedding argument as for \eqref{skembed2}, 
there is a two-dimensional Brownian motion $W$ started at $0$ and a sequence $\upsilon_1, \upsilon_2, \ldots $ of stopping times for $W$ such that for $i\geq 1$,
$W_{\upsilon_i}=\tilde c_{i/n}$
and
$$\P[|\upsilon_{\lfloor tn \rfloor}-\sigma ^2 (1-\v{s_n}) t|\geq n^{-1/3}]\leq \mc O(t n^{-1/3}).  $$
Fix $t>0$. Since $\v{s_n}=\frac{\log n}{n}$, for $n$ sufficiently large,
$$\P[|\upsilon_{\lfloor tn \rfloor}-\sigma ^2  t|\geq 2n^{-1/3}]\leq \mc O(n^{-1/3}).  $$
Then by a union bound over $j=1,\ldots,\lfloor n^{1/4}t\rfloor$,
\begin{align}
\P\l[\exists j\leq \lfloor n^{1/4}t \rfloor : |\upsilon_{\lfloor j n^{3/4} \rfloor} - \sigma ^2 j n^{-1/4}| \geq 2 n^{-1/3}\r]
&\leq (n^{1/4}t) \mc O(n^{-1/3})\label{union bound} \\
&=\mc O(n^{-1/12}).\notag
\end{align}
Again by a union bound over $j$,
\begin{align} 
&\P\bigg[\exists j \leq \lfloor n^{1/4}t \rfloor:\,\sup\bigg\{|W_{\sigma ^2 j n^{-1/4}} - W_u |\notag\\
&\hspace{4pc}\-u\in [\sigma^2 j n^{-1/4}-2n^{-1/3},\sigma^2 (j+1) n^{-1/4}+2n^{-1/3}]\bigg\}\geq n^{-1/10} \bigg]\notag \\
&\hspace{2pc}\leq (n^{1/4} t) 2 \P\l[\sup \{|W_s-W_0|:s\in [0, 4n^{-1/3}]\}\geq \tfrac{1}{2}n^{-1/10}\r]\notag\\
&\hspace{2pc}\leq 4n^{1/4}t \exp (-n^{2/15}/128)=o(n^{-1/12}).\label{eq:W_sup}
\end{align}
Here, the last line follows by \eqref{eq:2d_BM_sup}.

Under the complement of the event of~(\ref{union bound}), for all $j<\lfloor n^{1/4}t\rfloor $,
$$ |\upsilon_{\lfloor j n^{3/4} \rfloor} - \sigma ^2 j n^{-1/4}| \leq 2n^{-1/3}
\mbox{ and }
|\upsilon_{\lfloor (j+1) n^{3/4} \rfloor} - \sigma ^2 (j+1) n^{-1/4}| \leq 2n^{-1/3},$$
which implies that for $i$ such that $jn^{-1/4}\leq i n^{-1}\leq (j+1)n^{-1/4}$, 
$$\upsilon_i\in \l[\sigma^2 j n^{-1/4} - 2n^{-1/3},\sigma^2 (j+1) n^{-1/4} + 2n^{-1/3}\r].$$ 
Hence combining~(\ref{union bound}) and~(\ref{eq:W_sup}),
$$\P\l[\exists i \leq \lfloor tn \rfloor : |\tilde c_{i/n} - W_{\sigma ^2 i/n}|\geq 2n^{-1/10}\r]=
\mc{O}(n^{-1/12}). $$

Our next step is to control $|\tilde c_s-\tilde c_{i/n}|$ during the interval $s\in[i/n, (i+1)/n]$. 
The distribution of the number of jumps made by $\tilde c$ on an interval $[i/n,(i+1)/n]$ is Poisson 
with parameter $(1-\v{s_n})\lambda$, where $\lambda$ is given by \eqref{kappa_n_lambda}, and the maximum jump size is $2\mc{R}_n$;
using ~(\ref{poisson tail}) with $\chi=(1-\v{s_n})\lambda$ and $k=\log n$ gives that
$$\P\l[\exists i\leq \lfloor tn \rfloor :\sup_{s\in [i/n,(i+1)/n]} |\tilde c_s - \tilde c _{i/n}| \geq (\log n)2 \mathcal R_n \r] =
o(n^{-1}). $$
Hence for $n$ large enough that $(\log n)2 \mathcal R_n\leq n^{-1/10}$, using \eqref{eq:W_sup} again to bound $|W_s-W_{\sigma^2 i/n}|$ during the interval $[\sigma^2 i/n,\sigma^2 (i+1)/n]$ we have
\begin{equation} \label{eq:W_coupling}
\P\l[\sup_{s\leq t} |\tilde c_s - W_{\sigma ^2 s}| \geq 4n^{-1/10}\r]=\mc{O}(n^{-1/12}).
\end{equation}

We now apply this coupling to $(c^1_t)_{\tau^{br}_k+(\log n)^{-c}\leq t \leq \tau^{br}_{k+1}}$ for each $k\geq 0$, and let
the caterpillar evolve independently of the Brownian motion
on each interval $[\tau^{\text{br}}_k, \tau^{\text{br}}_k+(\log n)^{-c}]$.

More precisely, 
let $(\tilde c ^k)_{k \geq 0}$ be an i.i.d.~sequence of pure jump processes with  $\tilde c ^k_0=0$ and rate of jumps from $y$ to $y+z$ given by $(1-\v{s_n})m^n(dz)$.
Let  $(W^k)_{k\geq 0}$ be an i.i.d.~sequence of 2-dimensional  Brownian motions started at $0$ and for each $k\geq 0$, couple $W^k$ and $\tilde c^k$ in the same way as above, so that for fixed $t>0$, for each $k\geq 0$,
\begin{equation} \label{eq:Wk_coupling}
\P\l[\sup_{s\leq t} |\tilde c^k_s - W^k_{\sigma ^2 s}| \geq 4n^{-1/10}\r]=\mc{O}(n^{-1/12}).
\end{equation}
Then by the Strong Markov property for the process $c^1$, we can couple $(\tilde c ^k,W^k)_{k\geq 0}$ and $c^1$ in such a way that
for $k\geq 0$ and $s\in [0, \tau_{k+1}^{br}-(\tau_k^{br}+(\log n)^{-c}))$,
$$
c^1_{s+\tau_k^{br}+(\log n)^{-c}}-c^1_{\tau_k^{br}+(\log n)^{-c}}=\tilde c^k_s.
$$
and $(\tilde c ^k,W^k)_{k\geq 0}$ is independent of $\big(\tau^{\text{br}}_k, (c^1_t-c^2_t)|_{[\tau^{\text{br}}_k,\tau^{\text{br}}_k +(\log n)^{-c})}\big)_{k\geq 0}$.

Let $B$ be another independent 2-dimensional Brownian motion started at 0. We now define a single Brownian motion $W$ by piecing together increments of $B$ and $(W^k)_{k \geq 0}$.
For $s<\sigma^2(\log n)^{-c}$, let $W_s=B_s+p$.
Then for $k\geq 0$, define the increments of $W$ on the time interval $[\sigma^2(\tau^{br}_k+(\log n)^{-c}),\sigma^2(\tau^{br}_{k+1}+(\log n)^{-c}))$ as follows.
For $s \in [0,\sigma^2(\tau^{br}_{k+1}-\tau^{br}_k))$, let
$$
W_{s+\sigma^2(\tau^{br}_k+(\log n)^{-c})}-W_{\sigma^2(\tau^{br}_k+(\log n)^{-c})}=W^k_s.
$$
Then $W$ is a Brownian motion independent of $\big(\tau^{\text{br}}_k, (c^1_t-c^2_t)|_{[\tau^{\text{br}}_k,\tau^{\text{br}}_k +(\log n)^{-c})}\big)_{k\geq 0}$, which implies that $W$ is independent of both $k^*$ and $(\tau^{\text{br}}_k)_{k\geq 1}$.

We now check that $W_t$ is close to $c^1_t$ for $t<h$. 
By \eqref{eq:Ek_distn},
$$\P\l[\tau^{\text{br}}_{k+1}-\tau^{\text{br}}_k \geq 1+(\log n)^{-c}\r]\leq n^{-\lambda}.$$ 
Hence applying \eqref{eq:Wk_coupling} with $t=1+(\log n)^{-c}$ for each $k\leq (\log n)^{9/8}$ and using \eqref{k_star_exp_bound}, we have that with probability at least 
$1- \mc O(e^{-\delta (\log n)^{1/8}})$, for $0\leq k\leq k^*$ and $t\in [\tau^{\text{br}}_k +(\log n)^{-c},\tau^{\text{br}}_{k+1})$, 
\begin{equation}\label{after_lognc}
\l|\l(c^1_t- c^1_{\tau^{\text{br}}_k +(\log n)^{-c}}\r)-\l(W_{\sigma^2 t}-W_{\sigma^2(\tau^{\text{br}}_k +(\log n)^{-c})}\r)\r|\leq 4n^{-1/10}.
\end{equation}
For each $k$, by \eqref{eq:2d_BM_sup}, 
\begin{align} 
&\P\l[\sup \l\{|W_{\sigma^2 t}-W_{\sigma ^2 \tau^{\text{br}}_k}|:t\in [\tau^{\text{br}}_k,\tau^{\text{br}}_k+(\log n)^{-c}]\r\}\geq \tfrac{1}{3}(\log n)^{-c/3}\r] \notag\\
&\hspace{5pc}\leq 4 \exp(-(\log n)^{c/3}/72 \sigma^2)\notag \\
&\hspace{5pc}=o\l((\log n)^{-r-\frac{9}{8}}\r), \label{bm_jump}
\end{align} for any $r>0$.
Hence, using \eqref{k_star_exp_bound} again,
\begin{align} \label{mult_bm_jumps}
&\P\l[\sum_{k=1}^{k^*}\,\sup\l\{|W_{\sigma^2 t}-W_{\sigma ^2 \tau^{\text{br}}_k}|:t\in [\tau^{\text{br}}_k,\tau^{\text{br}}_k+(\log n)^{-c}]\r\}\geq \tfrac{1}{3}(\log n)^{\frac{9}{8}-\frac{c}{3}}\r]\notag\\
&\hspace{2pc}\leq \P\l[k^*\geq (\log n)^{9/8}\r]+(\log n)^{9/8} o((\log n)^{-r-\frac{9}{8}})\notag\\
&\hspace{2pc}=o((\log n)^{-r}).
\end{align}
For $k\geq 0$, on the time interval $ [\tau^{\text{br}}_k, \tau^{\text{br}}_k+(\log n)^{-c}]$ the process $c^1_t$ is a pure jump process with rate of jumps from $y$ to $y+z$ given by $m^n(dz)$.
Hence using the same Skorohod embedding argument as for \eqref{eq:W_coupling}, we can couple 
$(c^1_{s+\tau^{\text{br}}_k}-c^1_{\tau^{\text{br}}_k})_{s\leq (\log n)^{-c}}$ with a Brownian motion $W'$ started at $0$ in such a way that
$$ 
\P\l[\sup_{s\leq (\log n)^{-c}} |(c^1_{s+\tau^{\text{br}}_k}-c^1_{\tau^{\text{br}}_k}) - W_{\sigma ^2 s}| \geq 4n^{-1/10}\r]=\mc{O}(n^{-1/12}).
$$

Applying \eqref{mult_bm_jumps} and \eqref{k_star_exp_bound}, it follows that
\begin{multline*}
\P\bigg[\sum_{k=1}^{k^*}\,\sup\l\{|c^1_{t}-c^1_{\tau^{\text{br}}_k}|:t\in [\tau^{\text{br}}_k,\tau^{\text{br}}_k+(\log n)^{-c}]\r\}\notag\\
\hspace{6pc}\geq \tfrac{1}{3}(\log n)^{\frac{9}{8}-\frac{c}{3}}+4n^{-1/10}(\log n)^{9/8}\bigg]\\
= \mc{O}((\log n)^{-r}).
\end{multline*}

The stated result follows by combining the above equation with \eqref{after_lognc}, \eqref{k_star_exp_bound} and \eqref{mult_bm_jumps}.
\end{proof}

\subsubsection{The branching caterpillar}
\label{sec:branching_cat_sec}

We now construct a branching process of caterpillars. We start from a single caterpillar and
allow it to evolve until the time $h$. We start two independent caterpillars from the locations of $c^1_h$ and $c^2_h$. Now
iterate. The independent caterpillars defined in this way will be indexed by points of 
$\mathcal U=\{\emptyset\}\cup\bigcup _{k=1}^\infty \{1,2\}^k$. More formally:
\begin{defn}[Branching caterpillar] \label{branching_defn}

Let $(\Pi_j)_{j\in \mathcal U}$ be a sequence of independent Poisson point processes on $\mathscr{X}$ with intensity measure \eqref{eq:ppp_intensity}. For $p\in \R^2$, we define $(\mathcal C _t (p,(\Pi_j)_{j\in \mathcal U}))_{t\geq 0}$ as a process on $\cup_{k=1}^\infty (\R^2)^k$ as follows. 
For $s>0$, let
\begin{align}\label{eq:Pij}
\Pi_j^s = \{(t-s,x,r,z_1,z_2,q,v):(t,x,r,z_1,z_2,q,v)\in \Pi_j\}.
\end{align}
Define $(p_j,t_j,h_j)$ inductively for $j\in \mc U$ by $p_\emptyset =p$, $t_\emptyset =0$ and
\begin{align*}
h_j&=t_j+h(p_j,\Pi_j^{t_j})\\
t_{(j,1)}&=t_{(j,2)}=h_j\\
p_{(j,1)}&=c^1_{h_j-t_j}(p_j,\Pi_j^{t_j})\\
p_{(j,2)}&=c^{2}_{h_j-t_j}(p_j,\Pi_j^{t_j}).
\end{align*}
Finally, define $\mathcal U(t)=\{j\in \mathcal U:t_j\leq t \leq h_j\}$ and 
$$\mathcal C _t (p,(\Pi_j)_{j\in \mathcal U})=(c_{t-t_j}(p_j,\Pi_j^{t_j}))_{j\in \mathcal U(t)}.$$
\end{defn}

In words, $\mc{U}(t)$ is the set of indices of the caterpillars that are active at time $t$, and $\mc{C}_t$ is the set of (positions of) those caterpillars. Note that we translate the time coordinates in \eqref{eq:Pij} to match our definition of a caterpillar, which began at time $0$.
The jumps in $\mc{C}_t$ occur at the time coordinates of events in $\cup_{j\in \mc U}\Pi_j$.

We now show that for any constant $a>0$, with high probability, the longest `chain' of caterpillars has length at most $a \log \log n+1$. For $k\in \N$, let $\mathcal U _k=\{\emptyset\}\cup \bigcup_{j=1}^k \{1,2\}^j$. 
\begin{lemma} \label{branching_in_loglog}
Fix $T>0$; then for any $r>0$, $a>0$,
$\P[\mathcal U (T)\not\subseteq \mathcal U_{\lfloor a\log \log n \rfloor}]=o((\log n)^{-r})$.
\end{lemma}
\begin{proof}
Fix $v\in\{1,2\}^{\lfloor a\log \log n \rfloor+1}$. Then by a union bound,
\begin{equation} \label{contained in k levels}
\P\l[\exists w\in \{1,2\}^{\lfloor a\log \log n \rfloor+1}\text{ s.t. }t_w\leq T\r]\leq 2^{\lfloor a\log \log n \rfloor+1}\P[t_v\leq T]. 
\end{equation}

Note that by Lemma \ref{lifetime}, $t_v = \sum_{i=1}^{\lfloor a\log \log n \rfloor+1}H_i +R$ where $(H_i)_{i \geq 1}$ are i.i.d.~with $H_1\sim \text{Exp}(\lambda \kappa_n)$ and $$\P\l[R\geq 3(a\log \log n+1 )(\log n)^{-1/4}\r]=\mc O((\log \log n) e^{-\delta (\log n)^{1/8}}).$$ Hence (if $n$ is sufficiently large that $3(a\log \log n+1)(\log n)^{-1/4}\leq T/2$),
if $Z'$ is Poisson with parameter $\lambda \kappa_n T/2$,
$$
\P[t_v\leq T]\leq \P[Z'\geq a\log \log n+1]+\mc O\l((\log \log n) e^{-\delta (\log n)^{1/8}}\r).
$$
We use~(\ref{poisson tail}) and combine with~(\ref{contained in k levels})
to deduce that, for any $r>0$,
$$
\P[\mathcal U (T)\not\subseteq \mathcal U_{\lfloor a\log \log n \rfloor}]= \P\l[\exists w\in \{1,2\}^{\lfloor a\log \log n \rfloor+1}\text{ s.t. }t_w \leq T\r]=o((\log n)^{-r}).
$$
This completes the proof.
\end{proof}

The next task is to couple the branching caterpillar to the rescaled dual of the {\slfvs}. Since we have 
expressed the dual as a deterministic function of the driving point process of events in Definition \ref{slfvs_dual_determ}, it is enough to
find an appropriate coupling of the driving events for the branching caterpillar and those of a
{\slfvs} dual. 

The idea, roughly, is as follows. 
Each `branch' of the branching caterpillar is constructed from an independent 
driving process. For each of these we should like to retain those events that affected the caterpillar, but we can discard the rest.
If two or more caterpillars are close enough that the events affecting them could overlap, to avoid having too many events in these regions we have to arbitrarily choose one caterpillar and discard the events affecting the others.
We then supplement these with additional events, appropriately
distributed to fill in the gaps and arrive at the driving Poisson point process for a {\slfvs} dual, with intensity as in \eqref{eq:ppp_intensity}.  
We will then check that the {\slfvs} dual corresponding to this point process coincides 
with our branching caterpillar, with probability tending to one as $n\rightarrow\infty$.

To put this strategy into practice we require some notation.
Let $\mathcal U_0=\mathcal U \cup \{0\}$. For $V\subset \mc U_0$ let $\max(V)$ refer to the maximum element of $V$ with respect to a fixed ordering in which $0$ is the minimum value (it does not matter precisely which ordering we use, but we must fix one). Given a sequence $(\Pi_j)_{j\in \mathcal U_0}$ of independent Poisson point processes on $\mathscr{X}$ with intensity measure \eqref{eq:ppp_intensity}, define a simple point process $\Pi$ as follows. Let
\begin{equation} \label{jdefn}
j(t,x)=\max \l(\l\{k\in \mathcal U(t): \exists i \in \{1,2\}\text{ with }|c^i_{t-t_k}(p_k,\Pi_k^{t_k})-x|\leq \mathcal R_n\r\}\cup \{0\}\r).
\end{equation}
Note that $j(t,x)=0$ corresponds to regions of space-time that are not near a caterpillar, so that for $(t,x,r,z_1,z_2,q,v)\in \Pi_0$, $\mc B_r(x)$ does not contain a caterpillar. Then we define
\begin{equation} \label{Pidefn}
\Pi = \bigcup\limits_{k\in \mc{U}_0}\l\{(t,x,r,z_1,z_2,q,v)\in \Pi_k\-j(t,x)=k\r\}.
\end{equation}

\begin{lemma} \label{lemma_build_pp}
$\Pi$ is a Poisson point process with intensity measure given by \eqref{eq:ppp_intensity}.
\end{lemma}

\begin{remark}
We defined the coupling \eqref{Pidefn} for each $n\in\N$. As such, in the proof of Lemma \ref{lemma_build_pp} we regard $n$ as a constant and we will not include it inside $\mc{O}(\cdot)$, etc.
\end{remark}

\begin{proof}
Let $\nu (dt,dx,dr,dz_1, dz_2,dq,dv)$ be the intensity measure given in \eqref{eq:ppp_intensity}.

Let $\mathcal B_0$ be the set of bounded Borel subsets of  $\R_+ \times \R^2 \times \R_{+} \times \mathcal B_1(0)^2\times [0,1]^2$; for $B\in \mathcal B_0$, let $N(B)=|\Pi\cap B|$ and for $j\in \mathcal U_0$, let $N_j(B)=|\Pi_j\cap B|$. 
Suppose $B=\cup_{i=1}^k B_i\in \mathcal B_0$ where for each $i$, $B_i=[a_i,b_i]\times D_i$ for some $a=a_1<b_1\leq a_2<\ldots <b_k=b$. Let $\mathcal B_R\subset \mathcal B_0$ denote the collection of such sets $B$. 
Note that $\Pi$ is a simple point process, and that therefore $\Pi$ is a Poisson point process with intensity $\nu$ if and only if
\begin{equation}\label{eqn:ppp_criterion}
\P\l[N(B)=0\r]=e^{-\nu(B)}
\end{equation}
for all $B\in\mc B_R$. (See e.g. Section 3.4 of \cite{Kingman1992}.)

For some $\delta>0$, assume that $b_i-a_i\leq \delta $, $\forall i$ (by partitioning the $B_i$ further if necessary). Since $B$ is bounded, $\exists$ $ d<\infty$ s.t. $|x|\leq d$ for all $(t,x,r,z_1,z_2,q,v)\in B$. We can write
\begin{align} \label{void_prob}
\P[N(B)=0]&=\P[\cap_{i=1}^k \{N(B_i)=0\}]\nonumber \\
&=\E \l[\prod _{i=1}^{k-1}\1_{\{N(B_i)=0\}}\P\bigg(N(B_k)=0\bigg|(\Pi _j(a_k))_{j\in \mathcal U_0}\bigg) \r]
\end{align}
where $\Pi _j(t):=\Pi_j |_{[0,t] \times \R^2 \times \R_{+} \times \mathcal B_1(0)^2\times [0,1]^2}$. 

For $j\in \mathcal U_0$, let $D^j_k=\{(x,r,z_1,z_2,q,v)\in D_k:j(a_k,x)=j\}$ and $B^j_k=[a_k,b_k]\times D^j_k$. Also let 
$$\tilde B _k =[a_k,b_k] \times \mc B _{d+3\mc R_n}(0) \times \R_{+} \times \mathcal B_1(0)^2 \times [0,1]^2, $$ 
and let $\mathcal V (t)=\cup_{s\leq t}\mathcal U(s)$.

For $t\in [a_k,b_k]$, if none of the caterpillars in $\mc B_{d+3\mc R_n}(0)$ move during the time interval $[a_k,t]$ then $j(a_k,x)=j(t,x)$ $\forall x\in \mc B_d(0)$; thus a point $(t,x,r,z_1,z_2,q,v)$ in $\Pi \cap B_k$ must be a point in $\Pi_j \cap B^j_k$ for some $j$, and vice versa. We can use this observation to relate $\{N(B_k)=0\}$ and $\cap_{j\in \mathcal U_0}\{N_j(B^j_k)=0\}$, as follows.

If $N(B_k)=0$ and $N_j(B_k^j)\neq 0$ for some $j\in \mathcal U_0$, then $D^j_k\neq \emptyset$ so  $j\in \mc V(a_k)\cup \{0\}$ (either $j=0$ or the caterpillar indexed by $j$ is alive at time $a_k$). Also after $a_k$ and before the point in $\Pi_j \cap B^j_k$, one of the caterpillars in $\mc B_{d+3\mc R_n}(0)$ must have moved, so there must be a point in $\Pi_l \cap \tilde B_k$ for some $l \in \mathcal V(b_k)$. 
Conversely, if $N_j(B^j_k)=0$ $\forall j\in \mathcal U_0$ and $N(B_k)\neq 0$, then there must be a point in $\Pi_l \cap \tilde B_k$ followed by either a point in $\Pi_0 \cap B_k$ or a point in $\Pi_{l'} \cap B_k$ for some $l,l' \in \mathcal V(b_k)$. Hence 
\begin{align} \label{symmetric_diff}
&\{N(B_k)=0\}\triangle (\cap_{j\in \mathcal U_0}\{N_j(B^j_k)=0\})\subset \l\{N_0(B_k)+
\sum_{l\in \mathcal V (b_k)}N_l (\tilde{B_k})\geq 2\r\}.
\end{align} 
Note that by the definition of a caterpillar in Definition \ref{caterpillar_defn}, for each $j\in \mc U$, $h(p_j,\Pi^{t_j}_j)\geq (\log n)^{-c}$. It follows that $\mc V (b_k)\subseteq \bigcup _{m=0}^{\lceil b_k (\log n)^c \rceil }\{1,2\}^m$. Also
if $J\subset \mathcal U_0$ with $|J|=K$ then $\sum_{j\in J}N_j (\tilde B_k)$ has a Poisson distribution with parameter $K\nu (\tilde{B_k})$, 
and since $b_k-a_k \leq \delta$, $\nu(\tilde B_k)\leq n^2 \pi(d+3\mathcal R_n)^2 \mu((0,\mathcal R])\delta$.
Hence for $Z'$ a Poisson random variable with parameter $(2^{2+b_k (\log n)^c}+1)\nu(\tilde B_k)=\mc O (\delta)$,
$$
\P\l[N_0(B_k)+\sum_{j\in \mathcal V (b_k)}N_j (\tilde{B_k})\geq 2\bigg|(\Pi _j(a_k))_{j\in \mathcal U_0}\r]\leq \P\l[Z'\geq 2\r]=\mc{O}(\delta^2).
$$
By \eqref{symmetric_diff}, we now have that
\begin{align*}
\P[N(B_k)=0|(\Pi _j(a_k))_{j\in \mathcal U_0}]&=\P[\cap_{j\in \mathcal U_0}\{N_j(B^j_k)=0\}]+\mc{O}(\delta^2)\\
&=\prod _{j\in \mathcal U_0} \exp (-\nu (B^j_k))+\mc{O}(\delta^2)\\
&=\exp(-\nu (B_k))+\mc{O}(\delta^2).
\end{align*}
Substituting this into \eqref{void_prob} and then repeating the same argument for $k-1,k-2,\ldots,1$,
\begin{align*}
\P[N(B)=0]&=\prod _{i=1}^k \exp(-\nu (B_k))+ \sum_{i=1}^k \mc{O}(\delta^2)\\
&= \exp (-\nu (B))+ k \mc{O}(\delta^2).
\end{align*}
By partitioning $B$ further, we can let $\delta t \rightarrow 0$ with $k=\Theta(1/\delta)$. 
It follows that $\P[N(B)=0]=\exp (-\nu(B))$.
By \eqref{eqn:ppp_criterion}, this completes the proof.
\end{proof}

It follows immediately from Lemma \ref{lemma_build_pp} that the collection of potential ancestral lineages in $(\mathcal P _t (p,\Pi))_{t\geq 0}$ has the same distribution as $\mc P^{(n)}(p)$, the rescaled {\slfvs} dual.
We now show that under this coupling the rescaled {\slfvs} dual and branching caterpillar coincide with high probability. 

We consider $(\mathcal C _t (p,(\Pi_j)_{j\in \mathcal U}))_{0\leq t\leq T}$ as a collection of paths as follows. 
The set of paths through a single caterpillar $(c _t (p,\Pi))_{t\leq h(p,\Pi)}$ with $k^*(p,\Pi)=k^*$ is given by $\{l^i\}_{i\in \{1,2\}^{k^*}}$, where $l^i(t)=c ^{1}_t (p,\Pi)$ for $t\in [0,(\log n)^{-c}]$ and for each $1\leq k \leq k^*$, $l^i(t)=c ^{i_k}_t (p,\Pi)$ for $t\in [\tau^\text{br}_{k-1}+(\log n)^{-c},(\tau^\text{br}_{k}+(\log n)^{-c})\wedge h(p,\Pi)]$.
Then the collection of paths through $(\mathcal C _t (p,(\Pi_j)_{j\in \mathcal U}))_{0\leq t\leq T}$ is given by concatenating paths through the individual caterpillars, i.e.~paths $l:[0,T]\to \R^2$ such that for some sequence $(u_m)_{m\geq 0}\subset \mc U$ with $u_{m+1}=(u_m,i_m)$ for some $i_m\in \{1,2\}$ for each $m$, for $t\in [t_{u_m},h_{u_m}]$, $l(t)$ follows a path through $(c_{t-t_{u_m}}(p_{u_m},\Pi^{t_{u_m}}_{u_m}))_t$ with $l(h_{u_m})=p_{u_{m+1}}$.

\begin{lemma} \label{coupling with bc}
Fix $T>0$. Let $(\Pi_j)_{j\in \mathcal U_0}$ be independent Poisson point processes with intensity measure \eqref{eq:ppp_intensity} and let $\Pi$ be defined from $(\Pi_j)_{j\in \mathcal U_0}$ as in \eqref{Pidefn}. Then $(\mathcal C _t (p,(\Pi_j)_{j\in \mathcal U}))_{0\leq t\leq T}$ and $(\mathcal P _t (p,\Pi))_{0\leq t\leq T}$, viewed as collections of paths, are equal with probability at least $1-\mc{O}((\log n)^{-1/4})$.
\end{lemma}
\begin{proof}
We shall use Lemma \ref{branching_in_loglog} with $a=(16\log 2)^{-1}$.
Writing, for $j\in \mathcal U$, $k^*(j)=k^*(p_j, \Pi_j^{t_j})$, the number of branching events in $c_{t-t_j}(p_j,\Pi_j^{t_j})$ before $h_j$, by a union bound over $\mc U_{\lfloor a\log \log n \rfloor }$ and \eqref{k_star_exp_bound},
\begin{align} 
\P[\exists j \in \mathcal U_{\lfloor a\log \log n \rfloor}: k^*(j)&\geq (\log n)^{9/8}]\leq 2^{2+a\log \log n}\mc O(e^{-\delta (\log n)^{1/8}})\notag\\
&=\mc{O}(e^{-\delta (\log n)^{1/8}/2}).\label{multi_kstar_bound}
\end{align}
Let $(\tau^{\text{br}}_k(j))_{k\geq 1}$ denote the sequence of branching events in $c_{t-t_j}(p_j,\Pi_j^{t_j})$, and similarly define $(\tau^{\text{type}}_k(j))_{k\geq 1}$ and $(\tau^{\text{over}}_k(j))_{k\geq 1}$ as in \eqref{tau_type_cat}. 
Note that $(\mathcal C_t)_{t\leq T}$ and $(\mathcal P_t)_{t\leq T}$ only differ as collections of paths if either a selective event affects a caterpillar during a time interval in which it ignores branching, or if two different caterpillars are simultaneously within $\mc R_n$ of some $x\in \R^2$ and so one of them is not driven by the pieced together Poisson point process $\Pi$.
More formally, if $(\mathcal C_t)_{t\leq T}$ and $(\mathcal P_t)_{t\leq T}$ differ as collections of paths then one or more of the following events occurs.
\begin{enumerate}
\item $\mathcal U (T)\not\subseteq \mathcal U_{\lfloor a\log \log n \rfloor}$ or $k^*(j)\geq (\log n)^{9/8}$ for some $j\in \mathcal U_{\lfloor a\log \log n \rfloor}$.
\item For some $j\in \mathcal U_{\lfloor a\log \log n \rfloor}$ and $k\leq (\log n)^{9/8}$, the event $E_1(j,k)$ occurs: one of the lineages
$c^{1}_{t-t_j}(p_j,\Pi_j^{t_j})$ and $c^{2}_{t-t_j}(p_j,\Pi_j^{t_j})$ is affected by a selective event 
in the time interval $[\tau^{\text{br}}_k(j),\tau^{\text{br}}_k(j)+(\log n)^{-c}]$.
\item For some $w\neq v\in \mathcal U_{\lfloor a\log \log n \rfloor}$, the event $E_2(v,w)$ occurs: there are $i_1,i_2\in \{1,2\}$ with $|c^{i_1}_{t-t_w}(p_w,\Pi_w^{t_w})-c^{i_2}_{t-t_v}(p_v,\Pi_v^{t_v})|\leq 2\mathcal R_n$ for some $t\leq T$.
\end{enumerate} 
Recall from \eqref{kappa_n_lambda} and \eqref{jump of size z} that selective events affect a single lineage with rate $\lambda \log n$. Hence for $k\in \N$ and $j\in \mathcal U$, $\P[E_1(j,k)]=\mc{O}((\log n)^{1-c})$. 

We now consider the event $E_2(v,w)$.
For $w\neq v \in \mc U$, let $i= \min \{j\geq 1:w_j \neq v_j \}$. Then let
$$ w \wedge v = \begin{cases} (w_1,\ldots , w_{i-1}) \quad \text{if }i\geq 2\\
\emptyset \quad \text{if }i=1.
\end{cases} 
$$
At time $h_{w\wedge v}$, either $\tau_{k^*(w\wedge v)}^{\text{type}}(w\wedge v)=\tau_{k^*(w\wedge v)}^{\text{over}}(w\wedge v)$ or $\tau_{k^*(w\wedge v)}^{\text{type}}(w\wedge v)=\tau_{k^*(w\wedge v)}^{\text{div}}(w\wedge v)$, in which case $|p_{(w\wedge v,1)}-p_{(w\wedge v,2)}|\geq (\log n)^{-c}$. 
Conditional on $|p_{(w\wedge v,1)}-p_{(w\wedge v,2)}|\geq (\log n)^{-c}$,
for $i_1$, $i_2\in \{1,2\}$, 
$$\l(c^{i_1}_{t-t_w}(p_w,\Pi_w^{t_w}),c^{i_2}_{t-t_v}(p_v,\Pi_v^{t_v})\r)_{t\in[t_w,h_w]\cap [t_v,h_v]\cap [0,T]}$$
is part of the pair of potential ancestral lineages of an excursion started at time $h_{w\wedge v}$ with initial displacement at least $(\log n)^{-c}$.
Hence by Lemmas~\ref{Pinteract} and \ref{caterpillar_overshoot},
$$\P[E_2(w,v)]=\mc{O}\l(\frac{\log \log n}{\log n}\r)+\mc O \l((\log n)^{\frac{21}{8}-c}\r)=\mc O \l((\log n)^{-3/8}\r) $$
since $c\geq 3$. By a union bound, and using Lemma~\ref{branching_in_loglog} and \eqref{multi_kstar_bound} it follows that
\begin{align*}
&\P\l[(\mathcal C_t)_{t\leq T} \neq \mathcal (P_t)_{t\leq T}\r]\\
&\hspace{2pc}\leq o((\log n)^{-1})+4(\log n)^{a\log 2 +{9/8}}\P[E_1(j,k)]
 +16(\log n)^{2a\log 2}\P[E_2(w,v)]\\
&\hspace{2pc}=\mc{O}\l((\log n)^{a\log 2 +\frac{9}{8}+1-c}\r)+\mc{O}\l((\log n)^{2a\log 2-3/8}\r)\\
&\hspace{2pc}=\mc{O}\l((\log n)^{-1/4}\r),
\end{align*}
by our choice of $a=(16\log 2)^{-1}$ and since $c\geq 3$.
\end{proof}

We are now ready to complete the proof of Theorem \ref{result d>1}.

\begin{proof}
(Of Theorem \ref{result d>1})
Set $c=4$. By Lemmas \ref{lemma_build_pp} and \ref{coupling with bc}, we have a coupling of the rescaled S$\Lambda$FV dual and the branching caterpillar under which the two processes are equal (as collections of paths) with probability at least $1-\mc{O}((\log n)^{-1/4})$.

We now couple $(\mathcal C _t (p,(\Pi_j)_{j\in \mathcal U}))_{0\leq t\leq T}$ to a branching Brownian motion with branching rate $\lambda \kappa_n$. 
Let $((W_t^j)_{t\geq 0},H_j)_{j\in \mathcal U}$ be an i.i.d.~sequence, where $(W^j_t)_{t\geq 0}$ is a Brownian motion starting at $0$ and $H_j\sim \text{Exp}(\lambda \kappa_n )$ independent of $(W^j_t)_{t\geq 0}$. 
For each $j\in \mc U$, we couple $(c_{t-t_j}(p_j,\Pi_j^{t_j}))_{t \in [t_j,h_j]}$ to $((W_t^j)_{t\geq 0},H_j)$ as in Lemmas \ref{lifetime} and \ref{single_caterpillar}.

For $j\in \mc U$, let $A_1(j)$ be the event that both $|(h_j-t_j)-H_j|\leq 3(\log n)^{-1/4}$ and for $i=1,2$ and $t\in [t_j,h_j]$,
$$\l|(c^i_{t-t_j}(p_j,\Pi_j^{t_j})-p_j)-W^j_{\sigma^2 (t-t_j)}\r|\leq 2(\log n)^{\frac{9}{8}-\frac{c}{3}}=2(\log n)^{-5/24}. $$
By Lemmas \ref{lifetime} and \ref{single_caterpillar}, for any $r>0$, for each $j\in \mc U$, $\P[A_1(j)] \geq 1-\mc{O}\l((\log n)^{-r}\r)$.
Hence, taking a union bound over $j\in \mathcal U_{\lfloor \log \log n \rfloor}$,
$$\P[\cap_{j\in \mathcal U_{\lfloor \log \log n \rfloor}} A_1(j)]\geq 1- \mc{O}((\log n)^{\log 2-r}). $$
Also, for $j\in \mc U$, define the event
\begin{align*}
&A_2(j)=\bigg\{\sup_{t\in [0,3(\log n)^{-1/4}]}|W_{\sigma^2 t}^j|\\
&\hspace{8pc}+\sup_{t\in [H_j-3(\log n)^{-1/4},H_j]}|W_{\sigma^2 t}^j-W_{\sigma^2 H_j}^j|\leq (\log n)^{-1/9}\bigg\}. 
\end{align*}
Then by another union bound over $\mathcal U_{\lfloor \log \log n \rfloor}$,
since for a Brownian motion $(W_t)_{t\geq 0}$ started at $0$, $\P \l[ \sup_{t\in [0,3(\log n)^{-1/4}]}|W_t| \geq \tfrac{1}{2}(\log n)^{-1/9}\r]=o((\log n)^{-r})$, we have that 
$$\P[\cap_{j\in \mathcal U_{\lfloor \log \log n \rfloor}} A_2(j)]\geq 1- \mc{O}((\log n)^{\log 2-r}). $$
By Lemma \ref{branching_in_loglog}, $\P[\mathcal U (T)\not\subseteq \mathcal U_{\lfloor \log \log n \rfloor}]=o((\log n)^{-r}) $.

Define a branching Brownian motion starting at $p$ with diffusion constant $\sigma^2$ from $((W_t^j)_{t\geq 0},H_j)_{j\in \mathcal U}$ by letting the increments of the initial particle be given by $(W^{\emptyset}_{\sigma^2 t})_{t\geq 0}$ until time $H_\emptyset$, when it is replaced by two particles which have lifetimes $H_1$ and $H_2$ and increments given by $(W^{1}_{\sigma^2 t})_{t\geq 0}$, $(W^{2}_{\sigma^2 t})_{t\geq 0}$ and so on.

If $\mathcal U (T)\subseteq \mathcal U_{\lfloor \log \log n \rfloor}$ and $A_1(j)\cap A_2(j)$ occurs for each $j\in \mc U_{\lfloor \log \log n \rfloor}$, each path in the branching caterpillar stays within distance $2(\log \log n+1)(\log n)^{-1/9}+2(\log \log n+1)(\log n)^{-5/24}$ of some path through the branching Brownian motion and vice versa.

Setting $r=\log 2 +1/4$ gives us a coupling between the branching caterpillar and branching Brownian motion (with diffusion constant $\sigma^2$ and branching rate $\kappa_n \lambda$) such that with probability at least $1-\mc O((\log n)^{-1/4})$, up to time $T$ each path in the rescaled {\slfvs} dual stays within distance $2(\log \log n)(\log n)^{-1/9}+2(\log \log n)(\log n)^{-5/24}$ of some path through the branching Brownian motion and vice versa. Finally, we need to couple this branching Brownian motion up to time $T$ with a branching Brownian motion with branching rate $\kappa \lambda$.
By \eqref{kappaconv}, $\kappa_n \to \kappa$ as $n\to \infty$, so this follows by straightforward bounds on the difference between the branching times and the increments of a Brownian motion during such a time.
\end{proof}

\bibliographystyle{plainnat}
\bibliography{confirmation}

\begin{thebibliography}{21}
\providecommand{\natexlab}[1]{#1}
\providecommand{\url}[1]{\texttt{#1}}
\expandafter\ifx\csname urlstyle\endcsname\relax
  \providecommand{\doi}[1]{doi: #1}\else
  \providecommand{\doi}{doi: \begingroup \urlstyle{rm}\Url}\fi

\bibitem[Barton(1993)]{barton:1993}
N~H Barton.
\newblock The probability of fixation of a favoured allele in a subdivided
  population.
\newblock \emph{Genetical Research}, 62\penalty0 (02):\penalty0 149--157, 1993.

\bibitem[Barton et~al.(2010)Barton, Etheridge, and V\'eber]{BEV2010}
N~H Barton, A~M Etheridge, and A~V\'eber.
\newblock A new model for evolution in a spatial continuum.
\newblock \emph{Electron. J. Probab.}, 15:\penalty0 162--216, 2010.

\bibitem[Barton et~al.(2013{\natexlab{a}})Barton, Etheridge, Kelleher, and
  V\'eber]{barton/etheridge/kelleher/veber:2013b}
N~H Barton, A~M Etheridge, J~Kelleher, and A~V\'eber.
\newblock Genetic hitchhiking in spatially extended populations.
\newblock \emph{Theoretical population biology}, 87:\penalty0 75--89,
  2013{\natexlab{a}}.

\bibitem[Barton et~al.(2013{\natexlab{b}})Barton, Etheridge, Kelleher, and
  V{\'e}ber]{barton/etheridge/kelleher/veber:2013a}
NH~Barton, AM~Etheridge, J~Kelleher, and A~V{\'e}ber.
\newblock Inference in two dimensions: allele frequencies versus lengths of
  shared sequence blocks.
\newblock \emph{Theoretical population biology}, 87:\penalty0 105--119,
  2013{\natexlab{b}}.

\bibitem[Billingsley(1995)]{billingsley:1995}
P~Billingsley.
\newblock \emph{{Probability and Measure}}.
\newblock Wiley, 1995.

\bibitem[Cherry(2003)]{cherry:2003}
J~L Cherry.
\newblock Selection in a subdivided population with local extinction and
  recolonization.
\newblock \emph{Genetics}, 164\penalty0 (2):\penalty0 789--795, 2003.

\bibitem[Durrett and Z\"ahle(2007)]{durrett/zahle:2007}
R~Durrett and I~Z\"ahle.
\newblock On the width of hybrid zones.
\newblock \emph{Stochastic Processes and their Applications}, 117\penalty0
  (12):\penalty0 1751--1763, 2007.

\bibitem[Etheridge et~al.(2015)Etheridge, Freeman, and Straulino]{EFS2015}
A~Etheridge, N~Freeman, and D~Straulino.
\newblock The {Brownian} net and selection in the spatial {Lambda-Fleming-Viot}
  process.
\newblock \emph{arXiv preprint arXiv:1506.01158}, 2015.

\bibitem[Etheridge(2008)]{E2008}
A~M Etheridge.
\newblock Drift, draft and structure: some mathematical models of evolution.
\newblock \emph{Banach Center Publ.}, 80:\penalty0 121--144, 2008.

\bibitem[Etheridge and Kurtz(2014)]{EK2014}
A~M Etheridge and T~G Kurtz.
\newblock {Genealogical constructions of population models}.
\newblock \emph{arXiv preprint arXiv:1402.6724}, 2014.

\bibitem[Etheridge and V\'eber(2012)]{etheridge/veber:2012}
A~M Etheridge and A~V\'eber.
\newblock The spatial {Lambda-Fleming-Viot} process on a large torus:
  genealogies in the presence of recombination.
\newblock \emph{Ann. Appl. Probab.}, 22\penalty0 (6):\penalty0 2165--2209,
  2012.

\bibitem[Etheridge et~al.(2014)Etheridge, V\'eber, and Yu]{EVY2014}
A~M Etheridge, A~V\'eber, and F~Yu.
\newblock Rescaling limits of the spatial {Lambda-Fleming-Viot} process with
  selection.
\newblock \emph{arXiv preprint arXiv:1406.5884}, 2014.

\bibitem[Fisher(1937)]{fisher:1937}
R~A Fisher.
\newblock {The wave of advance of advantageous genes}.
\newblock \emph{Ann. Eugenics}, 7:\penalty0 355--369, 1937.

\bibitem[Kallenberg(2006)]{kallenberg:2006}
Olav Kallenberg.
\newblock \emph{Foundations of modern probability}.
\newblock Springer Science \& Business Media, 2006.

\bibitem[Kingman(1992)]{Kingman1992}
J~F~C Kingman.
\newblock \emph{Poisson processes}.
\newblock Oxford university press, 1992.

\bibitem[Krone and Neuhauser(1997)]{krone/neuhauser:1997}
S~M Krone and C~Neuhauser.
\newblock {Ancestral processes with selection}.
\newblock \emph{Theor. Pop. Biol.}, 51:\penalty0 210--237, 1997.

\bibitem[Maruyama(1970)]{maruyama:1970}
T~Maruyama.
\newblock On the fixation probability of mutant genes in a subdivided
  population.
\newblock \emph{Genetical research}, 15\penalty0 (02):\penalty0 221--225, 1970.

\bibitem[Mueller et~al.(2008)Mueller, Mytnik, and
  Quastel]{mueller/mytnik/quastel:2008}
C~Mueller, L~Mytnik, and J~Quastel.
\newblock Small noise asymptotics of traveling waves.
\newblock \emph{Markov Processes and Related Fields}, 14:\penalty0 333--342,
  2008.

\bibitem[Neuhauser and Krone(1997)]{neuhauser/krone:1997}
C~Neuhauser and S~M Krone.
\newblock {Genealogies of samples in models with selection}.
\newblock \emph{Genetics}, 145:\penalty0 519--534, 1997.

\bibitem[V\'eber and Wakolbinger(2015)]{veber/wakolbinger:2013}
A~V\'eber and Anton Wakolbinger.
\newblock {The spatial {Lambda-Fleming-Viot} process: an event-based
  construction and a lookdown representation}.
\newblock \emph{Ann. Inst. H. Poincaré Probab. Statist.}, 51\penalty0
  (2):\penalty0 570--598, 2015.

\bibitem[Whitlock(2003)]{whitlock:2003}
M~C Whitlock.
\newblock Fixation probability and time in subdivided populations.
\newblock \emph{Genetics}, 164\penalty0 (2):\penalty0 767--779, 2003.

\end{thebibliography}

\end{document}